\documentclass[12 pt]{amsart}
\usepackage{amssymb,amsfonts,amsthm, color, algorithm, algpseudocode, float}
\usepackage{paralist}
\usepackage{graphicx}
\usepackage{mathrsfs}
\DeclareSymbolFontAlphabet{\mathrsfs}{rsfs}
\usepackage{bm}
\newtheorem{thm}{Theorem}[section]

\newtheorem{lem}[thm]{Lemma}
\newtheorem{prop}[thm]{Proposition}

\theoremstyle{definition}
\newtheorem{de}[thm]{Definition}
\theoremstyle{remark}
\newtheorem{rem}[thm]{Remark}
\numberwithin{equation}{section}

\def\cT{{\mathcal T}}
\def\cX{{\mathcal X}}
\def\cY{{\mathcal Y}}
\def\cB{{\mathcal B}}

\def\al{\alpha}

\def\ep{\epsilon}

\def\de{\delta}

\def\ga{\gamma}

\begin{document}
\title[computations of geometric ergodicity]{Numerical computations of
  geometric ergodicity for stochastic dynamics}

\author{Yao Li}
\address{Yao Li: Department of Mathematics and Statistics, University
  of Massachusetts Amherst, Amherst, MA, 01002, USA}
\email{yaoli@math.umass.edu}
\author{Shirou Wang}
\address{Shirou Wang: Department of Mathematical and Statistical Sciences, University of Alberta, Edmonton, Alberta, Canada T6G2G1}
\email{shirou@ualberta.ca}

\keywords{Stochastic process, stochastic differential equations, geometric ergodicity, coupling method.}

\begin{abstract}
A probabilistic approach to compute the geometric convergence rate of a stochastic process is introduced in this paper. The goal is to quantitatively compute both the upper and lower bounds for rate of the exponential convergence  to the  stationary distribution of a stochastic dynamical system.
By applying the coupling method,
we derive an algorithm  
which does not rely on the discretization of  the infinitesimal generator. In this way,  our approach works well for many high-dimensional examples. 
We apply this algorithm to the random perturbations of both iterative maps and  differential equations.
We show that 
the rate of geometric ergodicity of a random perturbed system can, to some extent, reveal the degree of chaoticity of the underlying  deterministic dynamics.
Various SDE models  including the ones  with degenerate noise or living on the  high-dimensional state space are also explored. 
\end{abstract}
\maketitle

\section{Introduction}
In this paper, we consider the stochastic processes arising from the random perturbations of deterministic dynamical systems.  The dynamics of such a stochastic process, say ${\bm X} = \{X_{t}\}$, is a combination of a random diffusion and a  deterministic dynamics.  The rate of the ergodicity of ${\bm X},$ i.e., the speed of convergence of the law of $X_{t}$ to the invariant distribution, is a significant quantity  closely related to the spectral gap of the infinitesimal generator of ${\bm X}$, especially when ${\bm X}$
is reversible. From an applied viewpoint, knowing the speed of convergence is very useful to the sampling, uncertainty
quantification, and sensitivity analysis \cite{dobson2019using, johndrow2017error,
  mitrophanov2005sensitivity}.

However, the ergodicity of a stochastic process is difficult to study in a quantitative way. Methods  based on functional inequalities only work for a limited class of problems such as the over-damped Langevin dynamics \cite{bakry1985diffusions,
  holley1986logarithmic, lelievre2016partial}. The probabilistic approach, on the other hand,
although being ``softer'' and more applicable, usually does not give a precise bound in most of the existing results.
For instance, by constructing a Lyapunov function and establishing the minorization condition for a certain ``small set'', one can easily deduce the geometric ergodicity  \cite{hairer2010convergence, hairer2011yet, meyn2012markov}. Nevertheless, the rate
of geometric ergodicity obtained in this way is far from being optimal. In most cases, we only know that the exponential convergence rate to the steady
state is $\log\rho$ 
for some $\rho < 1$, 
but $\rho$  is usually too close to $1$ to be useful in practice.

The computational study of the ergodicity, on the other hand, is far
from being mature. 
While one can compute the
eigenvalues of the discretized infinitesimal generator for low-dimensional problems (1D or 2D)  as  discussed
in \cite{iacobucci2019convergence, pavliotis2013, risken, roussel2018spectral}, it  does not work well if ${\bm X}$
lives in a higher dimensional state space. 
One can obtain the convergence rate through the computation of the
correlation decay of a test function by the Monte Carlo simulation. However, as discussed in \cite{li2015stochastic}, the
correlation (or auto-correlation) has small expectation and large
variance, which results in an unrealistic requirement of large amount
of samples in the real simulations. In addition, the selection of test
functions is very subjective.  

The main goal of this paper is to propose a coupling approach, a powerful tool that has been used in many rigorous and computational studies \cite{bou2018coupling, eberle2019couplings, jacob2017unbiased, lindvall2002lectures,		lindvall1986coupling}, to numerically compute the geometric ergodicity.   
Traditionally,  the coupling method is mainly used in the theoretical study of stochastic dynamics. This is partially because in  computations,  a numerically simulated trajectory only approximate the real trajectory at discrete times with certain accuracy. As a result, on a continuous state
space, two numerical trajectories can easily ``miss'' each other even if the actual trajectories have already been coupled together. We solve this  by using the maximal coupling whenever two trajectories are sufficiently close and develop a corresponding numerical  scheme. By applying to various examples,
we show that our numerical coupling algorithm works well for the random perturbed iterative maps, the stochastic differential equations with non-degenerate diffusions, as well as the high-dimensional oscillators. Also, it can be well-adapted to certain systems with degenerate diffusions with some extra computational cost.  

A secondary goal of our study is to reveal how the geometric ergodicity of the perturbed stochastic system is related to the complexity of its underlying deterministic dynamics.  Applying to the random perturbed circle maps with distinct chaotic properties,  we show that the rate of geometric ergodicity, or heuristically the spectral property, can reveal, in some sense, the mixing property of the unperturbed  circle maps. For example, as the noise magnitude decreases, the rate of geometric ergodicity drops ``quickly" when the unperturbed  dynamics is ergodic but not mixing; while it drops ``dramatically" when the underlying dynamics admits a stable periodic orbit; see Section 4 for more details. Our simulation also shows that for the slow-fast systems,  a larger time scale separation between the slow and fast dynamics can enhance the geometric convergence rate when the random noises are added.  This can be explained by some heuristic arguments with numerical evidence.

The paper is organized as follows. Section 2 provides the necessary probability and dynamical system backgrounds.  Results serving as the theoretical basis of this paper are also presented and proved.   In section 3, various coupling mechanisms and our  numerical algorithms are described.  In Section 4,  by representative examples on the circle, we study the connection between the geometric ergodicity and the chaotic properties of the deterministic dynamics. In section 5, examples of stochastic differential equations with various deterministic or random structures are numerically studied. We conclude this paper in Section 6 with some further discussions and potential works.

\section{Preliminary}

\subsection{Markov process and geometric ergodicity} 
Throughout this paper, let $E$ be a state space, which can be   $\mathbb{R}^{k}, \mathbb{T}^{k}$, or a subset of $\mathbb{R}^{k},$ endowed with $\sigma$-field
$\mathcal{B}.$ 
Consider a Markov process {\boldmath $X$}$=\{X_{t};t\in\cT\}$ on $(E,
\mathcal{B}),$ where  $\cT$  can be 
$\mathbb{R}_{\ge0}$,$\mathbb{Z}_{\ge0}$, or $h \mathbb{Z}_{\ge0}:=\{0,h,2h,...\}$ for $h > 0.$ Let
 $\{P^{t}(x, A); x\in E, A\in\cB, t\in\cT\}$ be the transition probabilities of ${\bm X},$ i.e., for any $t\in\cT,$ $P^{t}(\cdot, A)$ is a measurable function for each fixed $A \in \mathcal{B}$,  and 
$P^{t}(x, \cdot)$ is a probability measure  for each fixed $x \in E$ such that
\[P^t(x,\cdot)=\int_E P^s(x,dy)P^{t-s}(y,\cdot),\quad 0\le s\le t.\]
In the following, for simplicity we denote the Markov process as  ${\bm X}=\{X_t\},$  the transition probabilities as $\{P^t\}$ when no ambiguity arises.

Given a Markov process ${\bm X}$ with  initial distribution  $\mu,$ for any $t\in\cT,$ $\mu P^t$ is the distribution of ${\bm X}$ at time $t$ such that 
\[
  \mu P^{t}(A) = \int_{E} P^{t}(x,A) \mu(
  \mathrm{d}x),\quad\forall A\in\cB.
\]
In particularly, $\mu$ is called {\it invariant} if $\mu P^t = \mu, \forall t\in\cT.$ 
 A Markov process {\boldmath $X$} is said to be {\it ergodic} if  it admits a unique
invariant (probability) distribution $\pi$ such that for any $x\in E, A\in\mathcal B,$ \[|P^{t}(x,A) -\pi(A)|\to0,\quad t\to\infty\] 
For a reference measure $\phi$ on $(E,\cB),$  {\boldmath $X$} is said to be {\it $\phi$-irreducible}  if given any $x \in E$, $\phi(A) >0$ implies that $P^{t}(x, A) > 0$ for some $t > 0$.
Throughout this paper, we assume that the Markov process {\boldmath $X$} 
is ergodic with an invariant (probability) distribution  $\pi.$ It is not hard to see that  {\boldmath$X$} is
$\pi$-irreducible.

The emphasis of this paper is the
geometric  ergodicity. An ergodic  Markov process  ${\bm X}$
is said to be {\it geometrically ergodic} with rate $r > 0$ if for $\pi$-a.e.
$x \in E,$
$$
\limsup_{t \rightarrow \infty}  \frac{1}{t} \log ( \|P^{t}(x,\cdot) -
\pi \|_{TV} ) = - r,
$$ 
where  $\|\mu-\nu\|_{TV}:=2\sup_{A\in\cB}|\mu(A)-\nu(A)|$ is the total variation distance between probability measures on $(E,\cB).$  A Markov process ${\bm X}$
is said to be {\it geometrically  contracting} with rate $r > 0$ if for $\pi\times\pi$-almost every initial pairs $(x,y)\in E\times E,$ 
it holds that 
\[
\limsup_{t \rightarrow \infty} \frac{1}{t} \log ( \| P^{t}(x,\cdot) -P^{t}(y,\cdot) \|_{TV}) = - r .\]

It is easy to see that the geometric ergodicity implies the geometric contraction. Since we already
assume the existence of an invariant probability measure, the
uniqueness of it directly follows
 from the geometrically contracting property. 
On the other hand, in the case of geometric contraction, one usually has estimate
 \[\|P^t(x,\cdot)-P^t(y,\cdot)\|_{TV}\le R(x,y) e^{-r t}\]
for a prefactor $R(x,y)$. It may happen that the prefactor  $R(x,\cdot)$
is too large to be integrable with respect to $\pi$, i.e., 
\[\int_E R(x,y)\pi(dy)=\infty,\] so that the
geometric convergence to the invariant measure $\pi$ may not be achieve at the same rate $r>0.$

\subsection{Coupling of Markov processes} 
In this paper, we investigate the geometric ergodicity of Markov processes through the coupling approach. This section serves as the theoretical background of it. We first  recall the  coupling of measures.  Let $\mu$ and $\nu$ be two probability measures on $(E,\cB)$. A
{\it coupling} of $\mu$ and $\nu$ is a probability measure on
$E \times E$ whose the  first and second
marginals are respective $\mu$ and $\nu.$ 
There is a well-known inequality showing that  the  total variation distance between $\mu$ and $\nu$ is bounded by  the difference of random variables realizing them.
To be specific, let $X$ and $Y$ be  random variables with  respective distributions $\mu$ and $\nu.$  Then (see, for instance, Lemma 3.6.  in \cite{aldous1983random})
\begin{eqnarray}\label{coup-ineq}
\|\mu - \nu \|_{TV} \leq 2\mathbb{P}[ X \neq Y].
\end{eqnarray}

Let ${\bm X}=\{X_{t};t\in\mathcal T\}$ and ${\bm Y}=\{Y_{t};t\in\mathcal T\}$
 be two stochastic  processes on $(E,\cB).$  A {\it coupling} of
{\boldmath $X$} and {\boldmath $Y$} is
a stochastic  process $({\bm X}, {\bm Y}) = \{(\mathcal{X}_{t}, \mathcal{Y}_{t}); t\in\mathcal T\}$ on
$E\times E$  such that
\begin{itemize}
	\item[(i)] The first and second marginal processes 
	$\{\mathcal{X}_t\}$ and $\{  \mathcal{Y}_{t}\}$ are respective copies of {\boldmath $X$} and {\boldmath $Y$};

	\item[(ii)] If $s\in\cT$ be such that  $\mathcal{X}_{s} = \mathcal{Y}_{s}$,  then $\mathcal{X}_{t} = \mathcal{Y}_{t}$ for all $t\ge s$.
\end{itemize}
The first meeting time of $\cX_t$ and $\cY_t,$ denoted as $\tau_{c}: = \inf_{t\ge0} \{\mathcal{X}_{t} = \mathcal{Y}_{t} \},$ is called the {\it coupling time}. 
A coupling $({\bm X}, {\bm Y})$ is said to be {\it successful} if the coupling time is almost surely finite, i.e., $\mathbb P[\tau_c<\infty]=1.$
Throughout this paper, we consider the couplings of two ergodic Markov processes, ${\bm X}$ and ${\bm Y},$ with a common  transition probabilities $\{P^t\}$ and a (unique) invariant (probability) distribution $\pi.$  
A coupling $({\bm X}, {\bm Y})$ is said to be a {\it Markov coupling}
if $({\bm X}, {\bm Y})$ is a Markov process. 
A Markov coupling  $({\bm X}, {\bm Y})$ is further called {\it irreducible} if it is $(\pi\times\pi)$-irreducible.

\begin{lem}\label{cplm}
Let  ${\bm X}$ and ${\bm Y}$ be Markov processes with a common transition probabilities $\{P^t\}$ and respective initial distributions $\mu$ and $\nu$.  Then for any coupling $({\bm X}, {\bm Y}),$  we have 
\begin{eqnarray}\label{proc_coup_ineq}
\|\mu P^{t} - \nu P^{t}\|_{TV} \leq 2 \mathbb{P}[\tau_{c} > t].
\end{eqnarray}
\end{lem}
\begin{proof}
	By the definition of coupling time, $\mathcal X_t\neq\mathcal Y_t$ implies that $\tau_c> t.$ Note that 
	$\mu P^t$ (resp. $\nu P^t$) is the distribution of $\cX_t$ (resp. $\cY_t$), then \eqref{proc_coup_ineq} follows from  \eqref{coup-ineq}.
\end{proof}

The inequality \eqref{proc_coup_ineq} is the well-known {\it coupling inequality}. 
 A coupling $({\bm X}, {\bm Y})$  
is said to be  {\it optimal} if 
the equality in \eqref{proc_coup_ineq} is achieved for any $t>0$.
In the present paper, we numerically
estimate the rate of geometric ergodicity of ${\bm X}$ (or ${\bm Y}$) via \eqref{proc_coup_ineq}. In practice, it is unrealistic to compute the coupling times for all initial values. Instead, we will develop some theoretical arguments that enable us to extend the result from one initial value to almost all initial values.

\begin{lem}
\label{prop1}
Let $({\bm X}, {\bm Y})= \{(\mathcal{X}_{t}, \mathcal{Y}_{t})\}$ be an irreducible Markov coupling of Markov processes ${\bm X}$ and ${\bm Y}$.  Assume that there exists a pair of initial value $(x_0,y_0) \in E\times E$ and a constant $r_{0} >
0$ such that
\begin{equation}\label{coup_time}
 \mathbb{E}_{(x_{0}, y_{0})}[e^{r_{0}\tau_{c}}] < \infty.
\end{equation}
Then \eqref{coup_time} holds for $(\pi\times\pi)$-almost all initial values.
\end{lem}
\begin{proof}
Suppose the lemma does not hold. Then there exists a measurable set $A\subseteq E\times E\backslash\{(x,x):x\in E\}$ with $(\pi\times\pi)(A)>0$ such that for any pair $(x,y)\in A,$ 
\begin{eqnarray}\label{large-tail}
 \mathbb{E}_{(x,y)}[ e^{r_{0}\tau_{c}}]  = \infty.
\end{eqnarray} 
By the irreducibility, there exists $T>0$ such that 
$\mathbb{P}_{(x_{0},
	y_{0})}[ ( \mathcal{X}_{T}, \mathcal{Y}_{T}) \in A] > 0. $
Then by \eqref{large-tail}, together with the Markov property, we have
 \begin{eqnarray*}
  \mathbb{E}_{(x_{0}, y_{0})}[ e^{r_{0}\tau_{c}}] &\geq& \int_{(\mathcal{X}_{T}, \mathcal{Y}_{T})\in A}e^{r_0\tau_c}d\mathbb P_{(x_0,y_0)} \\
 &\ge&\mathbb{P}_{(x_{0},
 	y_{0})}[ ( \mathcal{X}_{T}, \mathcal{Y}_{T}) \in A]\cdot\int_{(\mathcal{X}_{T}, \mathcal{Y}_{T})\in A}e^{r_0\tau_c}d\mathbb P_T\\
 &=&\mathbb{P}_{(x_{0},
 	y_{0})}[ ( \mathcal{X}_{T}, \mathcal{Y}_{T}) \in A]\cdot\mathbb{E}_{\mu}[e^{r_{0}(\tau_{c}-T)}]=\infty,
\end{eqnarray*}
where $\mathbb P_T$ is the conditional probability measure of  $\mathbb P_{(x_0,y_0)}$ conditioning on $ (
\mathcal{X}_{T}, \mathcal{Y}_{T}) \in A$, and
$\mu$ is the distribution of
$(\mathcal X_T,\mathcal Y_T)$ conditional on $A$. This contradicts with
\eqref{coup_time}. 
\end{proof}

One problem with Lemma \ref{prop1} is that many efficient couplings we shall use, 
such as the synchronous coupling and  reflection coupling (see Section 3 for the concrete meaning), are not
irreducible. On the other hand, although the independent coupling (i.e.,  the two marginal processes are updated independently all the time) brings about the irreducibility,
it is usually not efficient for the coupling process. In fact, 
most stochastic processes in
$\mathbb{R}^{k}$ (e.g.,  a strong-Feller process), including all the
numerical examples in this paper, are {\it non-atomic}, which means that 
any two independent trajectories of ${\bm X}$, say $X^{1}_{t}$ and
$X^{2}_{t}$, satisfy $\mathbb{P}[ X^{1}_{t+1} = X^{2}_{t+1}\,|\,
X^{1}_{t} \neq X^{2}_{t}] = 0$ (without loss of generality, here we assume that $\mathcal{T} = \mathbb{Z}_{\ge0}$).
So the independent coupling of a non-atomic Markov process has zero probability of being coupled successfully in finite time.

To overcome this difficulty, we introduce the coupling with independent components. Still, without loss of generality, we assume $\mathcal{T} = \mathbb{Z}_{\ge0}$. {\it A coupling with independent components}
means that  at each step before being coupled, with a positive
probability (which tough can be very small),  the two marginal processes are
updated in an independent way.  The following lemma shows that a
coupling with independent components of a non-atomic Markov process is irreducible. 
Thus, we can use a mixture of the independent coupling and other more efficient
couplings to achieve both the irreducibility and the coupling efficiency.

\begin{lem}
\label{prop2}
Let  $({\bm X}, {\bm Y})=\{(\cX_t,\cY_t)\}$ be a coupling with
independent components of non-atomic Markov processes ${\bm X}$ and ${\bm Y}$. Then $({\bm X}, {\bm Y})$ is $(\pi \times \pi)$-irreducible.  
\end{lem}
\begin{proof}
It is sufficient to show that for any product set $A_{1} \times A_{2} \in
\mathcal{B} \times \mathcal{B}$ with positive $\pi\times \pi$ measure, there exists some $t_{0} \in
\mathcal{T}$ such that $\mathbb{P}[ (\mathcal{X}_{t_{0}}, \mathcal{Y}_{t_{0}})
\in A_{1} \times A_{2}] > 0$. 

By the ergodicity, since $A_{1} \in \mathcal{B}$ has
positive $\pi$-measure, there exists $T_{1}>0$ such that $\mathbb{P}[ \cX_{t}
\in A_1 ] > 0$ for all $t > T_{1}$. Similarly, there exists $T_{2}>0$ such that $\mathbb{P}[ \cY_{t}
\in A_2] > 0$ for all $t > T_{2}$. Let  $t_{0} = \max \{ T_{1}, T_{2}\}+1$. Because  there is a positive probability that independent updates be chosen for $t = 0, 1, \cdots,
t_{0}$, and the Markov process  is non-atomic, we have $\mathbb{P}[ (\mathcal{X}_{t_{0}},
\mathcal{Y}_{t_{0}}) \in A_{1} \times A_{2}] > 0$.
\end{proof}

\begin{lem}\label{prop3}
Let $({\bm X}, {\bm Y})$ be a coupling with
independent components of non-atomic Markov processes ${\bm X}$ and ${\bm Y}$. Assume that there exist an initial value $x_{0} \in E$ and a constant $r_{0} >
0$ such that
\begin{eqnarray}\label{coup_time2}
  \mathbb{E}_{(x_{0}, \pi)}[e^{r_{0}\tau_{c}}] < \infty.
\end{eqnarray}
Then \eqref{coup_time2} holds for $\pi$-a.e. initial values $x\in E.$
\end{lem}
\begin{proof}
Suppose the lemma does not hold. Then there exists a measurable set $A \subseteq E$
with $\pi(A) > 0$ such that for any $x\in A,$
$$
  \mathbb{P}_{(x, \pi)}[ e^{r_{0}\tau_{c}}]  = \infty.
$$

Now let $\mathcal{X}_{0} = x_0$ and $\mathcal{Y}_{0} \sim \pi$.  By the irreducibility of ${\bm X}$, there exists a finite time $T > 0$ such that 
$P^T(x_0,A)>0.$
Denote $\lambda_{A}$ and $\pi_A$ as the conditional  measure of $P^{T}(x_0, \cdot)$ and $\pi$ on $A$, respectively.
Since
$({\bm X}, {\bm Y})$ is a coupling with independent components, the probability that $\mathcal{X}_{t}$ and $\mathcal{Y}_{t}$ remain being independent with each other for $t = 0, 1, \cdots, T$
is strictly positive. Since the Markov processes ${\bm X}$ and ${\bm Y}$ are non-atomic. Then
with probability $1,$ the independent updates will not make
${\bm X}$ and ${\bm Y}$ couple. Hence, 
there exists a positive number $\delta >
0$ such that
\[
\mathbb{P}_{(x_0,\pi)}[ (\mathcal{X}_{T}, \mathcal{Y}_{T}) \in C ] \geq 
\delta\cdot(P^T(x_0,\cdot)\times\pi)(C),\quad\forall\ C\subseteq E\times E.
\]
Applying the similar arguments as in Lemma \ref{prop1}, we have
\[
\mathbb{E}_{(x_{0}, \pi)}[ e^{r_{0}\tau_{c}}] \geq \delta\cdot P^T(x_0,A)\pi(A)\mathbb{E}_{\lambda_{A} \times \pi_A}[
 e^{r_{0}(\tau_{c}-T)}] = \infty.
\]
This contradicts to 
\eqref{coup_time2}. 
\end{proof}

It follows from Lemmata \ref{prop1}-- \ref{prop3} that
for any coupling with independent components, the
finiteness of $\mathbb{E}[e^{r_{0} \tau_{c}}]$ can be generalized from
one pair of initial values to almost all pairs. By the Markov inequality, we have
\begin{eqnarray*}
\mathbb P[\tau_c\ge t]\le\mathbb E[e^{r_0\tau_c}]e^{-r_0t}.
\end{eqnarray*}
Then together with the coupling inequality \eqref{proc_coup_ineq}, the finiteness of  $\mathbb{E}[e^{r_{0} \tau_{c}}]$ yields the geometric contraction/ergodicity.   However, the moment generating function $\mathbb{E}[e^{r_{0} \tau_{c}}]$ is difficult to compute in practice,  especially when $r_{0}$ is close to the critical value  $\sup\{r>0:\mathbb{E}[e^{r\tau_{c}}]<\infty\}.$ To overcome this, we turn to the estimate of the exponential tail of $\mathbb{P}[ \tau_{c} > t]$ instead. This is justified by the following Lemma.

\begin{lem}\label{equivalence}
For any initial distributions $\mu$ and $\nu$, assume that  for $r_0>0,$
\begin{equation}
	\label{exptail}
	\limsup_{t \rightarrow \infty}  \frac{1}{t} \log\mathbb
	P_{(\mu,\nu)}[ \tau_c>t ]  \leq - r_{0}.
\end{equation}
Then for any $\epsilon\in(0,r_0),$ it holds that 
\[
	\mathbb{E}_{(\mu, \nu)}[ e^{(r_{0} - \epsilon) \tau_{c}}] < \infty.
\]
\end{lem}

\begin{proof}
 By \eqref{exptail},  for  any $\epsilon\in(0,r_0),$ there exists $t_\epsilon<\infty$ such that for all $t\ge t_\epsilon,$ it holds that 
\begin{eqnarray*}
		\mathbb P_{(\mu, \nu)}[\tau_c>t ]\le e^{-(r_0 -\epsilon/2)t}. 
\end{eqnarray*}
Thus, for any $N>t_\ep,$
\begin{eqnarray*}
& & \mathbb{E}_{(\mu, \nu)}[ e^{(r_{0} - \epsilon) \tau_{c}}\cdot1_{\tau_c>N}]\\
&\le&\sum_{i=N}^{\infty}e^{(i+1)(r_0-\epsilon)}\mathbb P[ \tau_c=i ]\le\sum_{i=N}^{\infty}e^{(i+1)(r_0 -\epsilon)}e^{-(r_0-\epsilon/2)i}=e^{(r_0-\epsilon)}\sum_{i=N}^\infty e^{-i \epsilon /2},
\end{eqnarray*}
which goes to zero as $N$ goes to infinity. 
Hence, $\mathbb{E}_{(\mu,\nu)}[ e^{(r_{0} - \epsilon) \tau_{c}}] $ must be finite.
\end{proof}

Combine the above lemmata together,   
we have the following.
\begin{prop}\label{prop:summary}
Let $({\bm X}, {\bm Y})$ be a coupling with
independent components of non-atomic Markov processes ${\bm X}$ and ${\bm Y} $. 
\begin{itemize}
\item[(i)]Assume that there exist  an initial pair $(x_0,y_0)\in E \times E$ and $r_0>0$ such that 
\begin{equation*}
\limsup_{t \rightarrow \infty}  \frac{1}{t} \log\mathbb
P_{(x_0,y_0)}[ \tau_c>t ]  \leq - r_{0}.
\end{equation*}
Then for any $\epsilon\in(0,r_0)$, ${\bm X}$ (or ${\bm Y}$) is
geometrically contracting with rate $(r_0-\epsilon);$ 

\item[(ii)]  Assume that there exist $x_0\in E$ and $r_0>0$ such that  
\begin{equation*}
\limsup_{t \rightarrow \infty}  \frac{1}{t} \log\mathbb
P_{(x_0,\pi)}[ \tau_c>t ]  \leq - r_{0}.
\end{equation*} 
Then for any $\epsilon\in(0,r_0)$, ${\bm X}$ (or {\boldmath $Y$}) is
geometrically ergodic with rate $(r_0 - \epsilon)$. 
\end{itemize}
\end{prop}

\subsection{An upper bound of the geometric rate}
In general, the coupling inequality \eqref{proc_coup_ineq} only gives a lower bound of the geometric convergence/contraction rate.  We argue that in some cases, e.g., the random perturbation of a logistic map considered in Section 4.4,  the upper bound of the geometric ergodicity can be also estimated by using the first passage times because of the existence of the optimal coupling.

For sake of simplicity,  we consider the discrete-time Markov processes.
Recall that a coupling is said to be optimal if the equality in \eqref{proc_coup_ineq} holds for all times. 
It has been shown that for any two mutually singular probabilities
$\mu$ and $\nu$,  an optimal coupling with initial distribution
$\mu\times\nu$ exists and was explicitly constructed in \cite{griffeath1975maximal, pitman1976coupling}.

\begin{prop}
	\label{upperbound}
Let ${\bm X}=\{X_n; n\in\mathbb Z_{\ge0}\}$ and ${\bm Y}=\{Y_n; n\in\mathbb Z_{\ge0}\}$ be Markov processes on $E$ with initial conditions $X_0=x$ and $Y_0=y,$ respectively, where $x\neq y.$
Let  $\{(A_{n}, B_n)\}_{n=0}^\infty$ be a sequence of disjoint pairs of subsets in $E\times E$ such that 
$x\in A_{0}$, $y \in B_{0}.$ 
Assume that 
\[\rho:=
\limsup_{n\rightarrow
\infty} \frac{1}{n}\log\mathbb{P}[ \min \{ \eta_{x}, \eta_{y}\} > n]>0,
\]
where \[\eta_{x} = \min_{n > 0} \{ X_{n} \in A_{n}^{c}\},\quad \eta_{y} = \min_{n> 0} \{ Y_{n} \in 
B_{n}^{c} \}.\]
Then if ${\bm X}$ (or ${\bm Y}$) is geometrically contracting with rate $r>0$,
we have $r\le\rho.$
\end{prop}
\begin{proof}
Let $({\bm X}, {\bm Y})=\{(\mathcal X_n,\mathcal Y_n)\}$ be the optimal coupling of ${\bm X}$ and ${\bm Y}$. 
Then we have
\[
\|P^{n}(x,\cdot) - P^{n}(y,\cdot)\|_{TV} = 2 \mathbb{P}[ \tau_{c}> n],
\]
where $\{P^t\}$ is the common transition probabilities of ${\bm X}$ and ${\bm Y}$.
Note that at the
coupling time $\tau_{c},$ we have $\mathcal{X}_{\tau_{c}} =
\mathcal{Y}_{\tau_{c}}.$ This means that before time $\tau_{c}$, either $\mathcal{X}_{n}$ has
exited from $A_{n}$ or $\mathcal{Y}_{n}$ has exited from $B_{n}$, i.e., $\tau_c \leq n$  implies  $\min\{\tilde\eta_x,\tilde\eta_y\} \leq n.$ Here, $\tilde\eta_x,\tilde\eta_y$ are defined similarly as $\eta_x,\eta_y,$ but for the $\cX_n,\cY_n$  instead.  By noting  that for any $n\ge 0,$ $\cX_n$ (resp. $\cY_n$) has the same distribution as $X_n$ (resp. $Y_n$), we have
\[\mathbb P[\min\{\eta_x,\eta_y\}>n]=\mathbb P[\min\{\tilde\eta_x,\tilde\eta_y\}> n]<\mathbb P[\tau_c > n] .\]
This completes the proof.
\end{proof}

In Section 4.4, for a random  perturbed circle map  with a stable
2-periodic orbit, we shall give both  upper and  lower bounds of the
geometrically ergodic rate through the first exit times and the coupling times, receptively.

\subsection{Deterministic dynamics and random perturbations} 
Throughout this paper, by a discrete- or continuous-time deterministic dynamical system, we mean by iterating a map
\begin{eqnarray}\label{disc-map}
f:E\to E,
\end{eqnarray}
or an ordinary differential equation (ODE)
\begin{equation}\label{ODE}
\mathrm{d}Z_{t}/ \mathrm{d}t = g(Z_{t}),\quad t\in\mathbb R
\end{equation}
where $g$ is  a vector field on $E$  which  is locally Lipschitz
continuous.

In this paper, we mainly focus on the Markov processes arising from the random perturbations of a deterministic dynamical system. To be specific,  we shall consider 

\medskip

\noindent (i) The random perturbation of a discrete-time dynamics \eqref{disc-map} 
\begin{equation}
\label{discrete}
X_{n+1} = f(X_{n}) + \zeta_{n},
\end{equation}
where $\{\zeta_{n}\}$ 
are  independent random variables taking values in $E$  which will  be defined  specifically  in each particular situation;  

\medskip

\noindent (ii) The random perturbation of a continuous-time dynamics \eqref{ODE}  given by a stochastic differential equation (SDE) on $\mathbb R^k,$
\begin{equation}
\label{SDE}
\mathrm{d}X_{t} = g(X_{t}) \mathrm{d}t + \sigma(X_{t}) \mathrm{d}W_{t},
\end{equation}
where $\sigma(\cdot)$ is a $k\times k$ matrix-valued function and
$W_{t}$ is a Wiener process on $\mathbb{R}^{k}$. Here, $g$ and $\sigma$ are assumed to be  smooth enough to give a
well-defined solution $X_{t}$ for all $t > 0$. 

In the remainder of this section,  we briefly review a classical hierarchy  of chaotic properties of  deterministic dynamical systems, from the ergodicity to mixing. Readers may refer to \cite{katokBook95,walters82} for more details.  For sake of clarity and more fitting to the situation in Section 4,  we use intuitive examples of maps  on $\mathbb S^1,$ which are definitely not  essential restrictions. 

\medskip

\noindent {\it - Irrational rotations 
and ergodicity.} 
A deterministic map $f_\al:\mathbb S^1\to\mathbb S^1$ is said to be an irrational rotation (or quasi-periodic) if $f_\al x=x+\al$ (mod 1), where $\al$   is an irrational number.
Irrational rotation on $\mathbb S^1$ exhibits certain regular recurrent behavior that starts from any arbitrary initial point, the trajectory will visit 
any interval subsets in certain ``periodic" way. 
This is in fact  what the ergodic property says. 

For a deterministic dynamics $f$, the measure-theoretically chaotic property  is usually defined with respect to certain $f$-invariant measure $m,$ i.e., $m(fA)=m(A).$
An $f$-invariant measure $m$ is said to be ergodic if for any $\varphi\in C^0(E),$ $\psi\in L^1(m), $ it holds that 
\begin{eqnarray}\label{ergodic}
\dfrac{1}{n}\sum_{i=0}^{n-1}\int \varphi(f^ix)\psi(x)dm(x)\to \int \varphi dm\int\psi dm,\quad n\to\infty.
\end{eqnarray}
Another (and more well-known) characterization of ergodicity is through the Birkhoff ergodic theorem. For any $
\varphi\in L^1(m),$ for $m$-a.e. $x\in E,$ it holds that
\begin{eqnarray}\label{birkhoff}
\dfrac{1}{n}\sum_{i=0}^{n-1}\varphi(f^ix)\to m(A),\quad n\to \infty
\end{eqnarray}
if $f$ is ergodic with respect to $m$.
The expression \eqref{birkhoff} basically says that a typical trajectory visits any positive-measured set repeatedly  with frequency of the set measure.  The  irrational rotation on $\mathbb S^1$ is  ergodic with respect to the Lebesgue measure which is also the unique invariant measure.

\medskip

\noindent {\it - Expanding maps and 
mixing.}
The rotations on $\mathbb S^1$ only  
indicate a low-complexity of chaotic  properties since different orbits exhibits similar asymptotic behaviors. 
To characterize  more  non-trivial chaotic behaviors,  certain expanding properties are expected. 
A smooth circle map $f$
is said to be {\it  
expanding} if it always holds that $|f'|\ge1.$  An expanding map is further called uniform expanding if $|f'|$ is uniformly away from $1$. 
An expanding map often comes with the mixing property.
An $f$-invariant measure $m$ is said to be  {\it (strong) mixing} if for any $\varphi\in C^0(E),$ $\psi\in L^1(m), $ it holds that 
\begin{eqnarray}\label{mixing}
\int \varphi(f^nx)\psi(x)dm(x)\to \int \varphi dm\int\psi dm,\quad n\to\infty.
\end{eqnarray}
A mixing measure 
 is said to be {\em exponentially (resp. polynomially)
mixing}  if  \eqref{mixing} converges  in the exponential ({\it resp.} polynomial) way.  
It is well-known that the uniform  expanding maps are exponential mixing. For a general (non-uniform) expanding map however,    the exponential mixing property may be lost. A classical example illustrating this is the  expanding map with the only one neutral fixed point;
see Section 4.2 for more details.

\medskip

An intuitive way to understand the chaotic properties of ergodicity and mixing is to look at  how two different subsets (measure-theoretically) meet with each other under evolutions.  
Taking $\varphi=\chi_A, \psi=\chi_B$ where $A, B$ are two measurable subsets, respectively.  The ergodicity property \eqref{ergodic} ({\it resp.} mixing property \eqref{mixing})  yields
\begin{eqnarray*}\label{ergodic2}
\dfrac{1}{n}\sum_{i=0}^{n-1}m(f^{-i}A\cap B)\to m(A)m(B),\quad n\to\infty
\end{eqnarray*}
\[({\it resp.}\quad
m(f^{-n}A\cap B)\to m(A)m(B),\quad n\to\infty.)\]
It is not hard to see that the ergodicity property is mild  which  can be guaranteed if any two subsets can meet with each other in a ``regular" way (for instance, the irrational rotations on $\mathbb S^1$);
On the other hand, the mixing property requires a certain kind of ``stretching" of the system so that any two subsets can meet with each other eventually and forever. By  this, we see that the mixing is a stronger property than the ergodicity.

In section 4, we shall use four examples of circle maps with the degree of chaoticity  goes down from the exponentially/polynomially mixing to the ones without any mixing behaviors (which even exhibit contraction properties).  We observe that  although the geometric ergodicity usually holds when random noises are added, the rate can vary, as the noise vanishes, in different ways if the unperturbed  dynamics exhibits  distinct level of complexities. Again, we remark that  
the $\mathbb S^1$ setting is only for convenience. The scenario should be observed in more general state space. 

\subsection{Numerical scheme of SDEs}
In the real simulations, an SDE  is numerically computed at discrete times.  We usually choose a time step size $0<h \ll 1$ and consider the discrete-time trajectories $X_{0}, X_{h}, \cdots, X_{nh}, \cdots$. To
avoid confusion and make notations consistent, let ${\bm X} =
\{X_{t}; t \in
\mathbb{R}_{\ge0}\}$ be the true trajectories of the SDE, 
and $\bar{\bm X} = \{\bar{X}_{t}; t \in  h\mathbb{Z}_{\ge0} \}$ be
the trajectories of the numerical integrator. 
In addition, we denote
${\bm X}^{h} = \{X^{h}_{n}; n \in \mathbb{Z}_{\ge0}\}$ as the time-$h$ sample chain of
${\bm X}$ such that $X^{h}_{n} = X_{nh}$, and $\bar{\bm X}^{h} = \{\bar{X}^{h}_{n}; n \in\mathbb{Z}_{\ge0} \}$ as the time-$h$ sample chain of $\bar{\bm X}$
with $\bar{X}^{h}_{n} = \bar{X}_{nh}$. 

The
most commonly used numerical schemes of SDE \eqref{SDE}
is the Euler-Maruyama scheme
\[
\bar X_{(n+1)h} = \bar X_{nh} + g(\bar X_{nh})h + \sigma(\bar X_{nh}) \sqrt{h} N_{n},
\]
where $\{N_{n}\}$ are standard normal random variables independent for
each $n$. Note that the time-$h$ sample chain $\bar{X}^{h}_{n}$ fits the setting of discrete-time
random perturbed dynamics \eqref{discrete}
 \[
\bar{X}^{h}_{n+1} = \bar{X}^{h}_{n} + g(\bar{X}^{h}_{n})h + \sigma(\bar{X}^{h}_{n}) \sqrt{h}
N_{n}.
\]
The  Euler-Maruyama method can be improved to the Milstein method. The $1$D Milstein method reads as
\[X_{(n+1)h} = X_{nh} + f(X_{nh})h + \sigma(X_{nh}) \sqrt{h} N_{n}  +
	\frac{1}{2}\sigma(X_{nh})\sigma'(X_{nh})(N_{n}^{2} - 1)h.
\]
In particular, on any dimensions, the Euler-Maruyama method coincides with the
Milstein method if $\sigma(X_{t})$ is a constant matrix.

\medskip

Now, we recall the strong and weak approximations defined 
in \cite{kloeden2013numerical}. 
Let $T < \infty$ be a given finite time. If for $\ga>0,$
\[
\mathbb{E}[ | \bar{X}_{T} - X_{T} | ] \leq C(T) h^{\gamma}
\]
holds for all sufficiently small $h>0$, then we say that $\bar{{\bm X}}$ converges
{\it strongly}  to ${\bm X}$ with the order $\gamma$. 
Let $C^{\ell}_{P}$ denote the
space of $\ell$ times continuously differentiable functions with
polynomial growth rate for both the function itself and all the partial
derivatives up to order the $\ell$. If for $\ga>0$, any test function $g \in
C^{2(\gamma + 1)}_{P}$ and any given finite time $T$, we have
\[|\mathbb{E}[g(\bar{X}_{T})] - \mathbb{E}[g(X_{T})]| \leq C(T)
h^{\gamma},
\]
then we say that $\bar{\bm X}$ converges to ${\bm X}$ {\it weakly} with order $\gamma$.
It is well known that under suitable regularity conditions, the Euler-Maruyama scheme has strong  convergence with order $0.5$ and weak
convergence with order $1.0$. The Milstein scheme has strong convergence with order $1.0$ \cite{kloeden2013numerical}.

\section{Description of algorithm}
The main idea of this paper is to use the exponential tail of the coupling time distributions to numerically estimate the geometric ergodicity of a stochastic process. Assume that for a pair of initial values $(x_0,y_0)$ we have
$$
\mathbb{P}_{x_0,y_0}[\tau_c > t] \approx C e^{-r t},\quad\forall t\gg1.
$$
It follows from
Proposition \ref{prop:summary} that for almost every pair of initial values $(x,y)$,
$$
  \limsup_{t \rightarrow \infty} \frac{1}{t} \log ( \| P^{t}(x,\cdot) - P^{t} (y,\cdot)\|_{TV}) < - r.
$$
Replacing  $y$ by a sampling from the invariant distribution $\pi$, the numerical verification of geometric ergodicity is also  obtained by this approach.

Since this paper studies the coupling times in a numerical way,  we consider, for the sake of definiteness, 
the time-discrete Markov process ${\bm X}=\{X_n; n\in\mathbb Z_{\ge0}\}$
as it fits both cases of random perturbations of an iterative mapping and the time-$h$ sample chain of an SDE.  Note that here, the $n\in\mathbb Z_{\ge0}$ corresponds to the number of  iterations or numerical steps. For sake of differentiation and clarity, in the SDE setting, we will use $n _c$ to denote the numerical steps needed for  a successful coupling, which  of course  depends on the step size $h.$ The physical coupling time will be $\tau_c=n_ch.$

\subsection{Coupling methods}

Consider a Markov coupling $({\bm X}, {\bm Y})$. In the theoretical proof,  a coupling is usually done by making trajectories of both ${\bm X}$ and ${\bm Y}$ enter
a ``small set'' which 
satisfies the minorization condition \cite{meyn2012markov}. Numerically however, these couplings are not the most efficient ones. We will use a mixture of the following coupling methods to achieve the numerical coupling efficiently.

\medskip

\noindent{\it-  Independent coupling.} Independent coupling means that when running
the coupling process $({\bm X}, {\bm Y})$, the noise
terms in the two marginal processes $\cX_n$ and $\cY_n$ are independent until they are coupled. In
other words, we have
\[
  (\mathcal{X}_{n+1}, \mathcal{Y}_{n+1}) = \left ( f( \mathcal{X}_{n}) +
    \zeta^{1}_{n}, f( \mathcal{Y}_{n}) + \zeta^{2}_{n}\right ),
\]
where for each $n,$ $(\zeta^{1}_{n},\zeta^{2}_{n})$ is a pair of independent random variables.  In the theoretical 
studies, independent coupling is frequently used combined  with the renewal
theory to show the different rates of convergence to the invariant probability measure. In this paper, the independent coupling is to make the coupling process admit independent components so that Lemmata \ref{prop2} and \ref{prop3} are applicable.  

\medskip

\noindent{\it - Synchronous coupling.} Another commonly approach to coupling two
processes is the synchronous coupling. Contrary to the independent coupling for which the randomness in the two stochastic trajectories are totally unrelated, in the synchronous coupling, we always put the
same randomness to the both marginal processes until they are coupled, i.e., 
\[
  (\mathcal{X}_{n +1}, \mathcal{Y}_{n + 1}) = \left ( f( \mathcal{X}_{n}) +
    \zeta^{1}_{n}, f( \mathcal{Y}_{n}) + \zeta^{2}_{n}\right ),
\]
where $\zeta^{1}_{n} = \zeta^{2}_{n}$ for any  $n<n_c.$ The advantage of the synchronous coupling is that if the
deterministic part of the system 
already admits some kind of stability, then
$\mathcal{X}_{n}$ will approach to $\mathcal{Y}_{n}$ quickly when
the same noise is added each time \cite{arnold1995random}. The
synchronous coupling not only requires less assumptions on the random
terms, but also builds some potential connections between the random
dynamical system and stochastic differential equations; see Section 5.4
for a concrete example of the implementation of the synchronous coupling.

\medskip

\noindent{\it - Reflection coupling.} When the dimension of the state space is greater
than $2$, two Wiener processes will meet less often than the one/two dimensional  case. This makes the independent coupling less effective. The reflection coupling  will play a role instead.
As an example, take the Euler-Maruyama scheme of the SDE
\[
  \bar{X}^{h}_{n+1} = \bar{X}^{h}_{n} + f(\bar X_{n}^h)h + \sigma \sqrt{h} N_{n},
\]
where  $\sigma$ is an invertible constant matrix, and $N_{n}$ is a normal random variable with mean zero and
covariance matrix $\mathrm{Id}_k$. 
The reflection coupling means that we run the time-$h$ chain ${\bar{X}}^{h}_{n}$ as
\[
{\bar{X}}^{h}_{n+1} = {\bar{X}}^{h}_{n} +
  f({\bar{X}}^{h}_{n})h + \sigma \sqrt{h}  N_{n};
\]
while  run ${\bar{Y}}^{h}_{n}$ as 
\[{\bar{Y}}^{h}_{n+1} ={\bar{Y}}^{h}_{n} +
  f({\bar{Y}}^{h}_{n})h + \sigma \sqrt{h} P
  N_{n},
\]
where $P = I - 2 e_{n}e^\top_{n}$ is a projection matrix with
\[
  e_{n} = \frac{\sigma^{-1}({\bar{X}}^{h}_{n} - {\bar{Y}}^{h}_{n})}{\| \sigma^{-1}({\bar{X}}^{h}_{n} -
   {\bar{Y}}^{h}_{n})\| }.
\]
In other words,  the noise term is reflected against the hyperplane
that orthogonally passes the midpoint of the line segment connecting
${\bar{X}}^{h}_{n}$ and ${\bar{Y}}^{h}_{n}$.

Theoretically, it has been proved 
that  for the Brownian motions, the reflection coupling is optimal \cite{hsu2013maximal, lindvall1986coupling}, i.e, the equality in \eqref{proc_coup_ineq} is achieved for any $t>0$. It also works
well for many SDEs \cite{chen1997estimation, mufa1996estimation, eberle2011reflection, eberle2016reflection,lindvall1986coupling}, including the Langevin dynamics with  degenerate noise \cite{bou2018coupling,   eberle2019couplings}. The reflection coupling introduced above is also applicable to some non-constant $\sigma$ under suitable assumptions
\cite{lindvall1986coupling}. However, for a general non-constant
$\sigma(x)$, the ``true reflection'' is given by the Kendall-Cranston
coupling with respect to the Riemannian matrix
$\sigma^{T}(x)\sigma(x)$ \cite{cranston1991gradient,
  hsu2002stochastic, kendall1989coupled}, which is more difficult to
implement numerically. 

\medskip

\noindent{\it- Maximal coupling.} In the numerical simulations, the above
three couplings can only bring $\mathcal{X}_{n}$ close to
$\mathcal{Y}_{n}$. We still need a mechanism  to make $\mathcal{X}_{n+1}
 = \mathcal{Y}_{n+1}$ with certain  probability. The maximal
coupling aims to achieve this. It is derived to couple two trajectories as much as
possible at the next step, which is in fact modified from the now well-known
Doeblin coupling \cite{doeblin1938expose}.  We adopt the name ``maximal coupling'' from
\cite{jacob2017unbiased}. 

Assume that at certain step $n$, $(\mathcal{X}_{n} , \mathcal{Y}_{n})$
takes the value  $(x, y) \in E \times E$. Denote the probability measures associated with $f(x) + \zeta^1_n$ and $f(y) + \zeta^2_n$ by $\mu_x$ and $\mu_y$, respectively. Let $\nu_{x,y}$ be the ``minimum probability measure" of $\mu_x$ and $\mu_y$ such that
\[ 
\nu_{x,y}(A) = \dfrac{1}{\eta}\min \{\mu_x (A), \mu_y (A)\},
\]
where $\eta$ is a normalizer to make $\nu_{x,y}$ a probability
measure. At the next step, $(\mathcal{X}_{n + 1}, \mathcal{Y}_{n + 1})$  is sampled such that
\begin{itemize}
  \item[-] with probability $(1- \eta)$,
    \[
        \mathcal{X}_{n + 1} \sim \frac{1}{1-\eta}(\mu_x - \eta \nu_{x,y}), \quad \mathcal{Y}_{n + 1} \sim \frac{1}{1-\eta}(\mu_y - \eta \nu_{x,y})
    \]
    \item[-] with probability $\eta$,
    \[
    \zeta^1_n = \zeta^2_n \sim \nu_{x,y}.
    \]
\end{itemize}
In other words, ${\bm X}$ and ${\bm Y}$ are coupled if and only if the
two samples fall into a ``common future" simultaneously. We remark that
the classical version of Doeblin coupling requires that the two
trajectories enter a certain predefined ``small set"
simultaneously. Then a construction called the Nummelin split
guarantees them to be coupled with certain positive probability. However,  such a construction becomes
unnecessary when running the numerical simulations.  We can couple them whenever the probability distributions
of the next step have enough overlap.

\subsection{Numerical Algorithm}
We propose the following two numerical algorithms to estimate the
exponential tail of the coupling time for the rate of geometric
contraction/ergodicity. Both algorithms trigger the maximal coupling when distance between the two trajectories of a
coupling is smaller than a certain threshold. Since the maximal coupling
should have $O(1)$ successful rate when it is triggered, the threshold
$d$ in {\bf Algorithm 1} and {\bf Algorithm 2} should be proportional
to the standard deviation of distribution for the next step. The input of
{\bf Algorithm 1} is a pair of initial points $(x,y)$, and the output
is a lower bound of the geometric contraction rate of $\| P^n(x,\cdot)
- P^n(y,\cdot) \|_{TV}$. {\bf Algorithm 2} takes input of a point $x
\in E,$ and produces a lower bound of the convergence rate of $\|
P^n(x,\cdot) - \pi \|_{TV}$. In {\bf Algorithm 2}, we need to sample
from the invariant probability measure. This is done by choosing the
initial value of $\mathcal{Y}_0$ from a long trajectory of $X_n$, such
that $\mathcal{Y}_{0}$ is approximately sampled from the invariant
distribution $\pi$. 

Throughout this paper, coupling time distributions in the numerical examples are plotted in the log-linear plots with powers of $10$; while the slope of an exponential tail is computed by fitting $\log \mathbb{P}[ \tau_{c} > n]$ versus $n$ with a
linear function. Hence the slope of the coupling time distribution curves
equals $(\log 10)^{-1}$ times the corresponding output of {\bf
  Algorithm 1} or {\bf Algorithm 2}. 

\begin{algorithm}
\label{alg1}
\caption{Estimate geometric rate of contraction}
\begin{algorithmic}
\State {\bf Input: } Initial values $x, y \in E$
\State {\bf Output: } A lower bound of geometric rate of contraction $r > 0$
\State Choose threshold $d > 0$
\For {i = 1 to N}
\State $\tau_i = 0$,  $n = 0$, $(\mathcal{X}_n, \mathcal{Y}_n) = (x,y)$ 
\State Flag = 0
\While{Flag = 0}
\If {$|\mathcal{X}_n - \mathcal{Y}_n | > d $}
\State Compute $(\mathcal{X}_{n+1}, \mathcal{Y}_{n+1})$ using reflection coupling, synchronous coupling, or independent coupling
\State $n\gets n+1$
\Else 
\State Compute $(\mathcal{X}_{n+1}, \mathcal{Y}_{n+1})$ using maximal coupling
\If{$\zeta^1_n= \zeta^2_n\sim \nu_{x,y}$} 
\State Flag = 1
\State $\tau_i = n$
\Else
\State $n \gets n+1$
\EndIf
\EndIf
\EndWhile
\EndFor
\State Use $\tau_1, \cdots, \tau_N$ to compute $\mathbb{P}[\tau > n]$
\State Fit $\log \mathbb{P}[\tau >n]$ versus $n$ by a linear function. Compute the slope $- r$.
\end{algorithmic}

\end{algorithm}

\begin{algorithm}
\label{alg2}
\caption{Estimate convergence rate to $\pi$}
\begin{algorithmic}
\State {\bf Input: } Initial values $x\in E$
\State {\bf Output: } A lower bound of convergence rate $r > 0$ to $\pi$
\State Choose a threshold $d > 0$, another initial point $y \in E$, and a time step size $H$
\State Let $y_0 = y$
\For {i = 1 to N}
\State Let $X_0 = y_{i-1}$. Simulate $X_t$ for time $H$
\State $y_i \gets X_H$
\State $\tau_i = 0$,  $n = 0$, $(\mathcal{X}_n, \mathcal{Y}_n) = (x,y_i)$ 
\State Flag = 0
\While{Flag = 0}
\If {$|\mathcal{X}_n -  \mathcal{Y}_n| > d $}
\State Compute $(\mathcal{X}_{n+1}, \mathcal{Y}_{n+1})$ using reflection coupling, synchronous coupling, or independent coupling
\State $n\gets n+1$
\Else 
\State Compute $(\mathcal{X}_{n+1}, \mathcal{Y}_{n+1})$ using maximal coupling
\If{$\zeta^1_n = \zeta^2_n \sim \nu_{x,y}$} 
\State Flag = 1
\State $\tau_i = n$
\Else
\State $n\gets n+1$
\EndIf
\EndIf
\EndWhile
\EndFor
\State Use $\tau_1, \cdots, \tau_N$ to compute $\mathbb{P}[\tau > n ]$
\State Fit $\log \mathbb{P}[\tau > n ]$ versus $n$ by a linear function. Compute the slope $- r$.
\end{algorithmic}

\end{algorithm}

Since the geometric ergodicity implies the geometric contraction, in practice, it is sufficient only to run the {\bf
Algorithm 2} to detect the rate of geometric convergence/contraction
if the sampling from $\pi$ is possible. {\bf Algorithm 2} does not work
well if the convergence rate is too slow for a practical long time
trajectory to accurately represent samples from $\pi$. Theoretically,
one can still run
{\bf Algorithm 1} in this situation to get the geometric contraction
rate. However, a slow geometric convergence rate usually means the
geometric contraction rate is slow as well, which also affects the
implementation of {\bf Algorithm 1}.

\medskip

It remains to discuss the implementation of the maximal coupling. If
the probability density function of both $\mathcal{X}_{n+1}$ and
$\mathcal{Y}_{n+1}$ can be explicitly given, denoted by  $p^{(x)}(z)$
and $p^{(y)}(z)$ respectively, one can perform the maximal coupling by
comparing these two probability density functions. We adopt the
algorithm introduced in \cite{jacob2017unbiased,
  johnson1998coupling}. See  {\bf Algorithm 3} for the implementation
details. 

\begin{algorithm}[h]
\caption{Maximal coupling}
\label{max}
\begin{algorithmic}
\State {\bf Input:} $(\mathcal{X}_{t}, \mathcal{Y}_{t})$
\State {\bf Output:} $(\mathcal{X}_{t+1}, \mathcal{Y}_{t+1})$, and
$\tau_{c}$ if coupling is successful.
\State Compute probability density functions $p^{(x)}(z)$ and
 $p^{(y)}(z)$.
\State Sample $\mathcal{X}_{t+1}$ and calculate $W = Up^{(x)}(
\mathcal{X}_{t+1})$, where $U$ is a uniform random variable on
$(0, 1)$.
\If {$W \leq p^{(y)}(\mathcal{X}_{t+1})$}
\State $\mathcal{Y}_{t+1} = \mathcal{X}_{t+1}$, $\tau_{c} = t+1$
\Else
\State Sample $\mathcal{Y}_{t+1}$ and calculate $W' = Vp^{(y)}(
\mathcal{Y}_{t+1})$, where $V$ is a uniform random variable on
$(0, 1)$.
\While {$W' \leq p^{(x)}( \mathcal{Y}_{t+1})$}
\State Resample $\mathcal{Y}_{t+1}$ and $V$. Recalculate $W' = Vp^{(y)}(
\mathcal{Y}_{t+1})$.
\EndWhile
\State $\tau_{c}$ is  still undetermined. 
\EndIf
\end{algorithmic}
\end{algorithm}

\subsection{Some remarks}
As discussed in Section 2.2,  the reflection/synchronous coupling does not give an irreducible process in general, and we use a mixture of
independent coupling and reflection/synchronous coupling so
that the coupling has ``independent components''. To achieve this, at each step, we
generate an i.i.d. Bernoulli random variable $\Gamma$ with $\mathbb{P}[ \Gamma
= 1] = \beta>0$, which is independent of everything else. The independent coupling is chosen whenever $\Gamma = 1,$ and we use the reflection/synchronous coupling for otherwise. It then
follows from Lemmata \ref{prop2} and \ref{prop3} that the exponential
tail of the coupling time can be generalized to almost every initial
values. It is difficult to rigorously prove the effect of $\beta$. Our numerical simulation result (see section 5.3) shows that a smaller $\beta$ often corresponds to a higher convergence rate because the reflection coupling is more efficient.  

In practice, for all the examples we have tested and all the couplings we have
used, the exponential tails starting from
different initial values have the same rate. We believe that the requirement of the independent components
is only a technical limitation. Lemmata \ref{prop2} and
\ref{prop3} should hold true for a very general class of irreducible Markov
processes and couplings.

\section{Geometric ergodicity of time-discrete stochastic dynamics}
It has been observed that for qualitatively different deterministic
dynamical systems, their small random perturbations also have
qualitatively different asymptotic behaviors
\cite{lin2004convergence}. In this section, we  numerically perform four examples of random perturbations of 
deterministic maps on $\mathbb{S}^{1}$ with distinct  chaotic behaviors:  (1)  a uniformly expanding map; 
(2) an (almost) expanding  map admitting a neutral fixed point; 
(3) an irrational rotation; 
(4) a logistic map with a stable periodic orbit.
We note that the complexity  of dynamics is decreasing from (1)--(4).
For random perturbations of the above four dynamics,  the geometric convergence rates are computed and  compared under different noise magnitudes. Qualitative changes of the geometric convergence rates versus noises  are observed.  In
general,  as noise vanishes, the geometric convergence rate decreases in a slower way as the complexity of the underlying deterministic dynamics increases. Heuristic explanations of such changes are provided.

\subsection{Expanding circle maps}
\label{expanding}
Consider a deterministic dynamics  given by the iterative mapping $f: \mathbb S^{1} \rightarrow \mathbb S^{1}$:
\[
  f(x) = 2x + a \sin(2\pi x) \, (\mbox{mod } 1) .
\]
Note that for $a < 1/(2\pi)$, $f$ is uniformly expanding (i.e., $|f'| \geq 2(1-\pi a)>1$).
It has been known that the uniformly expanding map is  exponentially mixing with respect to an invariant probability measure with smooth density; see, for instance, \cite{viana97}.

Consider the  Markov
process ${\bm X}$  given by the random perturbation of $f$ as follows
\begin{equation}\label{exp1_per}
  X_{n+1} =f(X_{n}) + \epsilon \zeta_{n} \, (\mbox{mod } 1),
\end{equation}
where $\{\zeta_{n}\}$ are i.i.d. standard normal
random variables, and $\epsilon$ is the noise magnitude. 
In our simulations, we run {\bf Algorithm 2} with $N = 10^{8}$ samples
and collect the coupling times. For all the examples throughout this section,  the threshold $d$ of triggering the
maximal coupling is set as $2\epsilon$ because $\epsilon \zeta_{t}$ has a standard deviation $\epsilon$.   When the maximal coupling is
triggered, we compare the probability density function on the line $\mathbb{R}$
and then fold back to $\mathbb{S}^{1}$. Theoretically, this is smaller than the ``true maximal coupling'' for which the
coupling probability should add up all  the periodic images. However, it makes little difference here since $\epsilon\ll1.$

In Figure \ref{fig1}, the $\mathbb{P}[\tau_{c} > n]$ versus $n$ plots are demonstrated in the log-linear plot, where the noise magnitudes $\epsilon$ are chosen to
be $0.01,$ $0.02,$ $0.04,$ $0.06,$ $0.08,$ $0.1$ and $0.12$, respectively.  We
see that the coupling time distribution has exponential tails which gives the rate
of geometric ergodicity. Slopes of those exponential tails are obtained by fitting $\log
\mathbb{P}[\tau_{c} > n]$ versus $n$ using a linear function. The negative slope of the exponential tails
versus $\epsilon$ is demonstrated in the lower right panel of Figure
\ref{fig1}). It drop linearly with respect to the noise magnitude.  (Note that the log-linear plot uses the logarithm with
base $10$. Hence, the slopes of curves in Figure \ref{fig1} Left and Middle
are the corresponding outputs of {\bf Algorithm 2} multiplied by $(\log
10)^{-1}$. This applies to all numerical examples in this paper.) This is
expected because the threshold to trigger the maximal coupling is $2\epsilon$. 
Two trajectories need to be
$O(\epsilon)$-close in order to couple. If we assume that the trajectory of $f$ is
well-mixed, heuristically two trajectories should take
$O(\epsilon^{-1})$ time to be $O(\epsilon)$ close to each other.

\begin{figure}[htbp]
\centerline{\includegraphics[width = \linewidth]{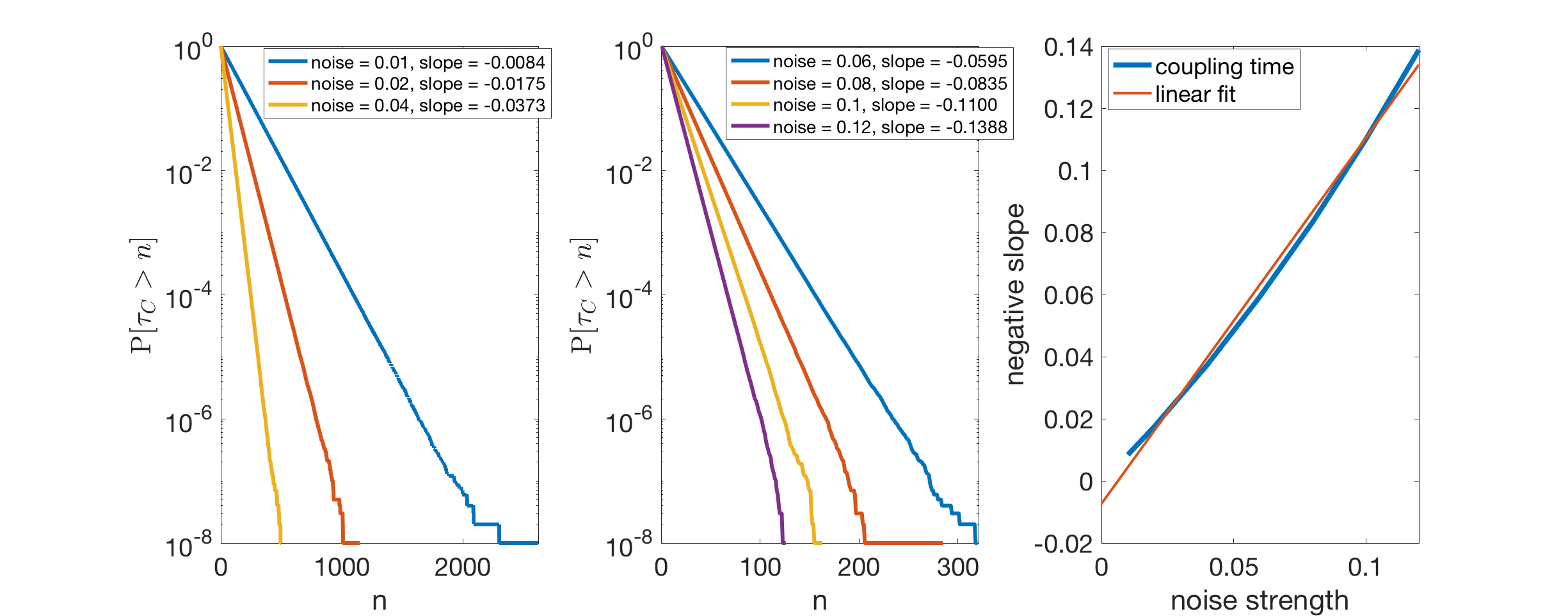}}
\caption{ Example in Section 4.1.  Left and Middle: 
The $\mathbb{P}[\tau_{c} > n]$ vs. $n$ with
 different noise magnitudes. Right: The negative slope of the exponential tail  vs. noise magnitudes, and the linear fit.}
\label{fig1}
\end{figure}

\subsection{Circle maps with neutral fixed point}
\label{neutral}
The second example is a circle map with a neutral fixed point. Consider 
$$
  f(x) = \left \{ 
\begin{array}{ll}
x + 2^{\alpha}x^{1+\alpha}, & 0 \leq x \leq \frac{1}{2};\\
2x - 1, &\frac{1}{2} < x < 1,
\end{array}
\right .
$$
where $0 < \alpha < 1$ is a parameter. Note  that $|f'|\ge1$ on $[0,1],$ 
and $|f'| = 1$ is achieved only at  $x = 0,$ i.e., 
$x=0$ is the (unique) neutral
fixed point. Thus, $f$ is not necessarily   
exponentially mixing.  In fact,  it has been shown that in this example, $f$
has the power-law mixing rate
$n^{1 - 1/\alpha}$ \cite{young1999recurrence}. 

Now, we consider the small random perturbation of $f$ given by  the Markov process ${\bm X}$ as follows
\[
   X_{n+1} =f(X_{n}) + \epsilon \zeta_{n} \, (\mbox{mod } 1),
\]
where $\{\zeta_{n}\}$ are i.i.d. standard normal random variables. 
Still, we  run {\bf Algorithm 2} with $N = 10^{8}$ samples and collect all the coupling times to compute the rate of geometric ergodicity. Noise magnitudes $\epsilon$ are chosen  the same as in Section 4.1. The $\mathbb{P}[\tau_{c} > n]$ versus $n$  are demonstrate in the log-linear plot in Figure \ref{fig2} (the left and middle panel). We see that the coupling time distribution still admits exponential tails, the slope of which
versus $\epsilon$ is computed and plotted in Figure \ref{fig2} (the right panel). Note  that despite a slower mixing rate (polynomial) of $f$,  the slope of the
exponential tail still drops linearly with respect to the noise magnitude, which is
same as the exponential mixing example in Section 4.1. This is because the slow mixing of $f$ is due to a longer return time from the very small neighborhood of the unique neutral fixed point $x=0$. A very small noise is already sufficient to ``shake'' the trajectories away from the neutral fixed point
to maintain a suitable mixing rate. Hence, the effect of slower-mixing
rate is hard to be observed unless the noise term becomes extremely
small. We refer to \cite{blumenthal2017lyapunov, blumenthal2018lyapunov}
for more recent theoretical results of similar maps with very small random perturbation.   

\begin{figure}[htbp]
\centerline{\includegraphics[width = \linewidth]{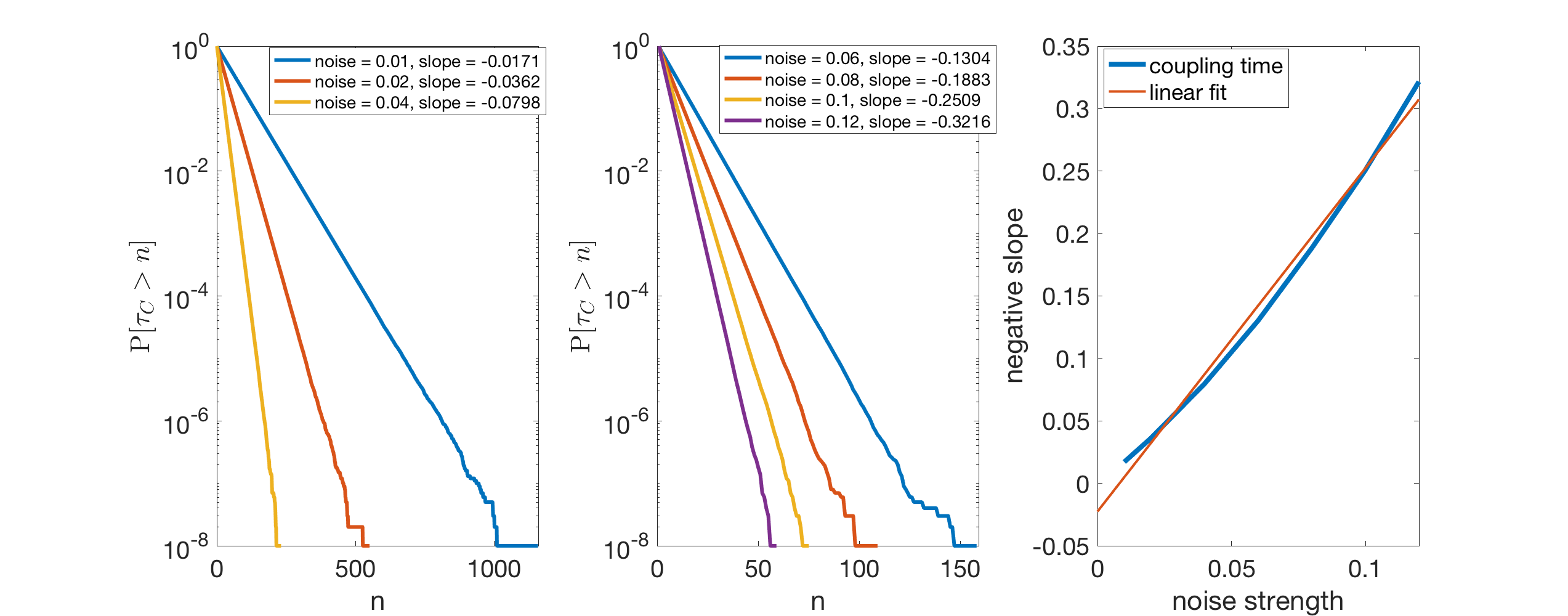}}
\caption{ Example in Section 4.2. Left and Middle:
  $\mathbb{P}[\tau_{c} > n]$ vs. $n$ with
 different noise magnitudes. Right:  The negative slope of the exponential tail  vs. noise magnitudes, and the linear fit.}
\label{fig2}
\end{figure}

\subsection{Irrational rotation (quasi-periodic)}
\label{quasi}
The third example is the irrational rotation on $\mathbb S^1$
\begin{eqnarray}\label{irrational}
 f(x) = x + \sqrt{2} \,(\mbox{mod } 1).
\end{eqnarray}
 Distinct from the previous two examples, for the irrational rotation \eqref{irrational},  there is NO  any stretching for the map $f$ (since $|f'|\equiv1$).  
Also,  every orbit of $f$ is dense going almost everywhere on $\mathbb S^1.$ Thus, $f$ is ergodic but not mixing. 

 Now, we consider the Markov process ${\bm X}$ given by 
$$
  X_{n+1} =f(X_{n}) + \epsilon \zeta_{n} \, (\mbox{mod } 1),
$$
where $\{\zeta_{n}\}$ are i.i.d. standard normal random variables. 
Still, the rate of geometric ergodicity are computed  by  running {\bf Algorithm 2} with $N = 10^{8}$ samples under different noise magnitudes
(Here, $\epsilon$ are chosen the same as
the previous two examples).
The $\mathbb{P}[\tau_{c} > n]$ versus $n$ plots are demonstrated
in the log-linear plot in Figure \ref{fig3}. We see that the coupling time distributions still exhibit exponential tails as the  the previous two examples.  However, in this example,  the slope (of the
exponential tail)  versus $\epsilon$ curve  drops super-linearly, instead of linearly as in the previous two examples,  as the noise magnitude decreases. We fit it by a quadratic polynomial function fairly well; see  the  right panel in Figure \ref{fig3}.
The heuristic reason for the $O(\epsilon^{2})$ slope is
the following. Without mixing, the only force that brings two
trajectories together is the diffusion, which takes $O(\epsilon^{-2})$
time to move $O(1)$ distance. Hence, one can expect two trajectories to be
``well mixed'' after $O(\epsilon^{-2})$ time.

\begin{figure}[htbp]
\centerline{\includegraphics[width = 1.2\linewidth]{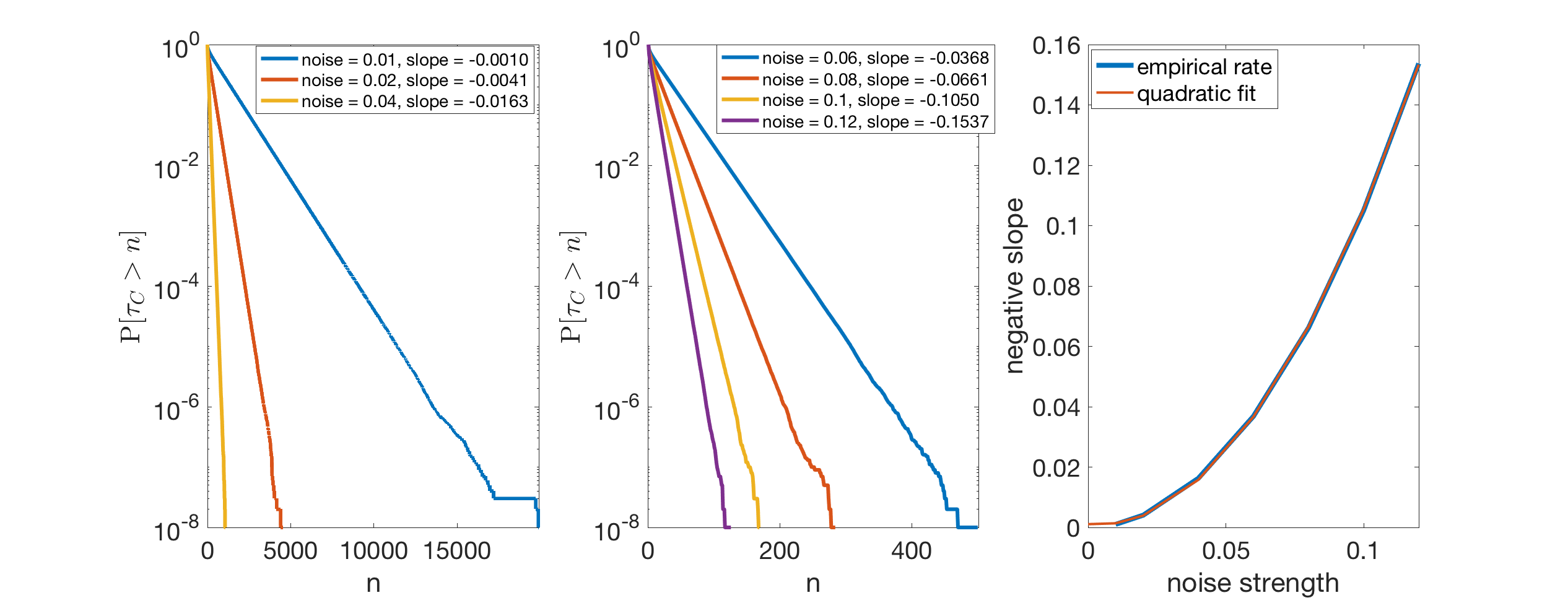}}
\caption{Example in Section 4.3. Left and Middle: $\mathbb{P}[\tau_{c} > n]$ vs. $n$ with
  different noise magnitudes. Right:  The negative slope of the exponential tail  vs. noise magnitudes, and the quadratic fit.}
\label{fig3}
\end{figure}

\subsection{Logistic map with stable periodic orbit}
\label{periodic}
The last example is from the logistic family 
\begin{eqnarray*}
f_{\lambda}=\lambda x(1-x)\  (\mbox{mod } 1) 
\end{eqnarray*}
where $0\le \lambda\le4. $ The logistic map was introduced as a demographic model  \cite{ausloos2006logistic} and 
has been well studied since then for its  manipulability and abundant dynamical phenomena. 
It has been known that 
for $\lambda$ between 2 and 3.56995 (approximately), the dynamics of $f_\lambda$ is simple. There is a periodic orbit, for which the period  doubles as  $\lambda$ increases,  attracting all the other trajectories.  However, for a typical  $\lambda$ beyond the critical value 3.5699,  the dynamics of  $f_\lambda$  goes into a chaotic regime.  
Any two trajectories will diverge no matter how close initially they are.
In this example, we choose the logistic map \begin{eqnarray*}
f:=f_{3.2}=3.2x(1-x)\  (\mbox{mod } 1).
\end{eqnarray*}
which admits a
2-periodic orbit $PQPQPQ\cdots,$ where $P=0.7995, Q=0.5130$ (approximately),  that attracts all the initial values in $(0,1)$.

Now,  we consider the Markov chain ${\bm X}$
\[
  X_{n+1} =f(X_{n}) + \epsilon \zeta_{n} \, (\mbox{mod } 1),
\]
where $\{\zeta_{n}\}$ are i.i.d. standard normal random variables. 
Still, we compute the rate of geometric ergodicity of ${\bm X}$ with different
noise magnitudes by running {\bf Algorithm 2} with $N = 10^{8}$
samples trajectories. A little bit different from the previous three examples, the noise magnitudes in this example are
chosen as $0.015,$ $0.02,$ $0.03,$ $0.04,$
$0.06,$ $0.08,$ and $0.1$, respectively. This is because in this example, the coupling is extremely slow which is hard to be observed numerically if the noise is too small. Slopes of the exponential tails of the
coupling times  are computed and demonstrated in Figure \ref{fig4} (blue lines). The lower right panel in Figure \ref{fig4} shows that the coupling becomes exponentially slow as the noise vanishes.  This is because the
trajectories start from the basin of the different periodic sequences $PQPQP\cdots$ and $QPQPQ\cdots$ need to ``overcome the attraction" from the corresponding periodic sequence  in order to meet.

 In addition to the lower bound, for this example, we also compute
the upper bound of the rate of geometric ergodicity through the first exit time. By transparent calculations, one finds that the basin of attraction of the periodic sequences
$PQPQPQ\cdots$ and  $QPQPQP\cdots$ are 
\[
  A = [0.110 \, , \, 0.312] \cup [0.688 \, , \, 0.890]\cup\cdots \quad  {\text{and}} \quad B
  = [0.313 \,, \, 0.688]\cup\cdots,
\]
respectively, i.e., a deterministic trajectory starts from $A$  converges to the
periodic sequence $PQPQPQ \cdots$, and  a deterministic trajectory
starts from $B$ converges to the periodic sequence $QPQPQP\cdots$. For each value of $\epsilon$ chosen above, we compute the first exit time $\eta_{PQ}$ of the coupling 
$({\bm X}, {\bm Y})$ starting from $A\times B$ as follows
\begin{align*}
  \eta_{PQ} &= \min \left \{ \inf_{n \geq 0} \{n \,|\, \mathcal{X}_{n} \notin A, n \mbox{ even, or
  } \mathcal{X}_{n} \notin B, n \mbox{ odd }\} , \right. \\
& \left . \inf_{n \geq 0} \{n \,|\, \mathcal{Y}_{n} \notin B, n \mbox{ even, or
  } \mathcal{Y}_{n} \notin A, n \mbox{ odd }\} \right \}.
\end{align*}
The Log-linear plots of $\mathbb{P}[ \eta_{PQ} >
n]$ versus $n$ are also demonstrated in Figure \ref{fig4} (red
lines). Still, we run $10^{8}$ samples. We see that  when $\epsilon$
is small, the distribution of the first exit time is also
exponentially small as the coupling times. By Proposition \ref{upperbound},
this gives an upper bound for the rate of geometric ergodicity.

  If let $S(\epsilon)$ and $\hat S(\epsilon)$ be the slopes of the exponential tails of the coupling time and first exit time under the $\epsilon$-noise perturbation,  then the large deviation theory tells  that 
$\epsilon^{2}\log(-\hat S(\epsilon))$ converges to a finite limit
as $\epsilon$ vanishes \cite{freidlin1998random}. This is confirmed by our numerical
simulations in Figure \ref{fig4} Lower Right (red crosses). In addition, the term $\epsilon^{2}\log(- {S}(\epsilon))$ also converges to a finite limit as well (blue dots). 

 \begin{figure}[htbp]
\centerline{\includegraphics[width = 1.2\linewidth]{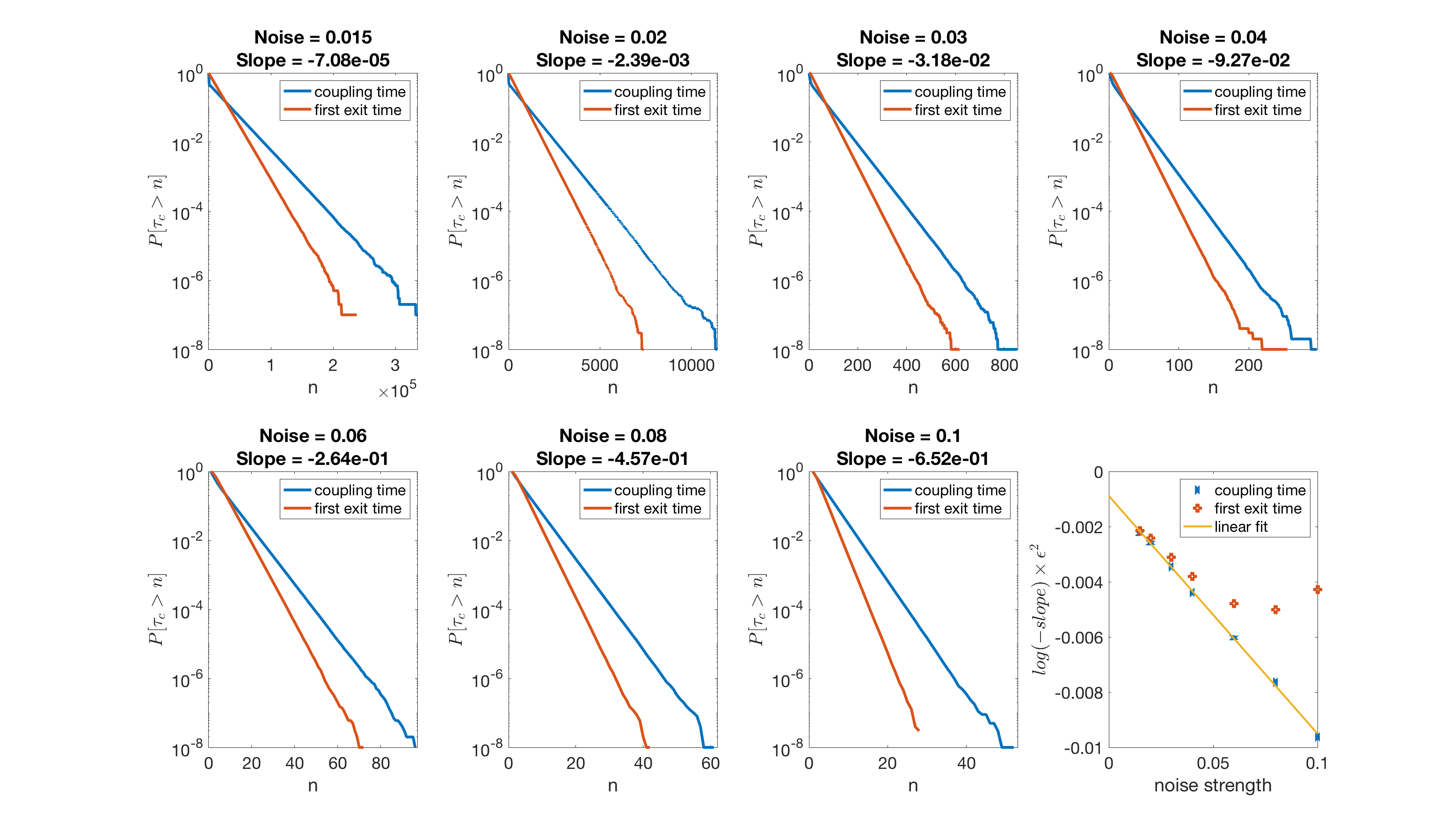}}
\caption{Example in Section 4.4. First 7 panels: Blue line:  $\mathbb{P}[\tau_{c} > n]$ vs. $n$ with different noise magnitudes. Red line:  $\log\mathbb{P}[\eta_{PQ} > n]$ vs. $n$ with
  different noise magnitudes. Lower right panel: Linear extrapolation of
  $\epsilon^{2}\log(- S(\epsilon))$ and $\epsilon^{2}\log(-
  \hat{S}(\epsilon))$, where $S(\epsilon)$ and $\hat S(\epsilon)$ are the slopes of exponential tail of
 the coupling time and the first exit time with respect to the noise magnitude $\epsilon,$ respectively.
}
\label{fig4}
\end{figure}

\section{Geometric ergodicity of stochastic differential equations}

\subsection{Numerical and analytical coupling times}
For SDEs, the first issue to address is the impact of numerical
approximations. As we know,  the numerical trajectories $\bar{X}_{t}$ of an SDE are only  approximations of the true trajectories $X_{t}$. Although the independent/synchronous/reflection coupling methods introduced in Section 3.1 can be analogously applied to the
 SDE setting \eqref{SDE},  the analytical
coupling mechanisms are different from the numerical ones. Two trajectories of
$X_{t}$ are coupled whenever they meet, without the need to trigger a maximal
coupling one step earlier. Such a difference makes the direct comparison of coupling times between $X_{t}$ and $\bar X_t$ difficult, if not impossible. To solve this,  for the time-$h$ sample chain of the true SDE, we apply the numerical coupling strategy as well, i.e., the maximal coupling is triggered when two trajectories are close to
each other. This enables us to compare the coupling times under the same
coupling mechanisms.

Now, applying the coupling strategy in {\bf Algorithm 1}, 
we construct a coupling
$({\bm X}^h, {\bm Y}^h)$ of the time-$h$ chain of \eqref{SDE} as follows:

\begin{itemize}
  \item[(i)]  When the maximal coupling is not triggered, $(\mathcal{X}^h_{n},
  \mathcal{Y}^h_{n})$  evolves according to the same coupling method
  (independent, reflection, or synchronous) as the one used by the numerical coupling  $(
  \mathcal{\bar{X}}^h_{n}, \mathcal{\bar{Y}}^h_{n})$
 for $n = 0, 1, \cdots;$
 
\item[(ii)] At each $t = nh, n\ge0$,
  check the distance between $\mathcal{X}^h_{n}$ and
  $\mathcal{Y}^h_{n}.$ Trigger the maximal coupling if and only
  if $| \mathcal{X}^h_{n} - \mathcal{Y}^h_{n} | < d$, where $d$ is the
  same threshold as in {\bf Algorithm 1};
  
\item[(iii)]If the maximal coupling is triggered at $t = nh$, perform the maximal
coupling with respect to the probability distribution of
$\mathcal{X}^h_{(n+1)}$ and $\mathcal{Y}^h_{(n+1)},$ respectively\footnote{If at step (ii), we already have $\mathcal{X}^h_{n} = \mathcal{Y}^h_{n}$. Then we just set 
	$\tau_c = nh,$ and the step (iii) will not be implemented. However, for
	strong Feller processes, this happens with zero probability.}.
\end{itemize}

\medskip

It is easy to see that $({\bm X}^h, {\bm Y}^h)$ is a coupling of the time-$h$
sample chain of the SDE \eqref{SDE}. 
We further assume the following for $({\bm X}^h, {\bm Y}^h)$ and $(\bar{\bm X}^h, \bar{\bm Y}^h)$, respectively. 
\begin{itemize}
  \item[\bf{(S1)}]  The numerical scheme used in {\bf Algorithm 1} is a strong approximation. More precisely,  for any finite $t>0,$ there exists a constant $C(t)>0$ such that 
\[
  \mathbb{P}[ | \bar{X}^{h}_{i} - X^{h}_{i} | > h^{p}] \leq C(t) h^{1 + \alpha},\quad i=0,1,...,\lfloor t/h\rfloor+1
\]
holds for some $p > 1/2$, $\alpha>0,$  and all sufficiently small $h > 0;$

\medskip

\item[\bf{(S2)}] For each $\boldsymbol z := \boldsymbol x - \boldsymbol y$, the probability density function of
  $Z := \mathcal{X}^{h}_{1} - \mathcal{Y}^{h}_{1}$
  ({\it resp.} $\bar{Z} := \mathcal{\bar{X}}^{h}_{1} - \mathcal{\bar{Y}}^{h}_{1}$)  given
  $\mathcal{X}^{h}_{0} = \boldsymbol x, \mathcal{Y}^{h}_{0} = \boldsymbol y$
  ({\it resp.} $\mathcal{\bar{X}}^{h}_{0} = \boldsymbol x, \mathcal{\bar{Y}}^{h}_{0} =
  \boldsymbol y$ ), denoted by $p_{\boldsymbol z}(Z)$ ({\it resp.} $\bar{p}_{\boldsymbol z}( \bar{Z})$), satisfies 
\[
  C_{b}h^{-k/2} e^{-(Z -\boldsymbol z)^\top \Sigma_{b}(Z - \boldsymbol z)/h}  \leq p_{\boldsymbol z}(Z) \leq
  C_{u}h^{-k/2} e^{-(Z - \boldsymbol z)^\top \Sigma_{u}(Z - \boldsymbol z)/h}
\]
\[({\it resp.}\quad
  \bar{C}_{b}h^{-k/2} e^{-(\bar{Z} - \boldsymbol z)^\top \bar{\Sigma}_{b}(\bar{Z} - \boldsymbol z)/h}  \leq \bar{p}_{\boldsymbol z}(\bar{Z}) \leq
  \bar{C}_{u}h^{-k/2} e^{-(\bar{Z} - \boldsymbol z)^\top \bar{\Sigma}_{u}(\bar{Z} - \boldsymbol z)/h} ),
\]
where $\Sigma_{u}, \Sigma_{b}$
({\it resp.}$\bar{\Sigma}_{u}, \bar{\Sigma}_{b}$) are positive definite $k\times k$ matrices,  and 
$C_{u}, C_{b}>0$ ({\it resp.} $\bar{C}_{u}, \bar{C}_{b}>0$) are constants in order $O(1)$;

\medskip

\item[\bf{(S3)}] The threshold $d$ to trigger the maximal coupling is in order $O(\sqrt h).$ To be specific, we set $d=2\epsilon\sqrt{h},$ where $\epsilon$ is 
the noise magnitude in \eqref{SDE}; 

\medskip

 \item[\bf{(S4)}] The probability density function of $X^{h}_{1}$ ({\it resp.} $\bar{X}^{h}_{1}$)
  conditioning on $X^{h}_{0} = \boldsymbol x$ ({\it resp.} $\bar{X}^{h}_{0} = \boldsymbol x$),
denoted by $f^{h}_{\boldsymbol x}$ ({\it resp.} $\bar{f}^{h}_{\boldsymbol x}$), changes continuously with respect to $h.$ More precisely, there exists a function $\varphi : \mathbb{R}_+ \to \mathbb{R}_+$ satisfying $\lim_{h\to0}\varphi(h)=0$ 
  such  that for all $\gamma > 1/2$ and the unit vector $\boldsymbol{v}\in\mathbb R^k,$ it holds that 
\[
  \| f^{h}_{\boldsymbol x} -  f^{h}_{\boldsymbol x + h^{\gamma} \boldsymbol{v}} \|_{L^{1}}
 <\varphi(h) \quad  (\mbox{resp. } \| \bar{f}^{h}_{\boldsymbol x} - \bar{f}^{h}_{\boldsymbol x +
   h^{\gamma} \boldsymbol{v}} \|_{L^{1}}
 <\varphi(h) ).
\]
In addition, the one-step transition
probability density function $\bar{f}^{h}_{\boldsymbol x}$ approximates
$f^{h}_{\boldsymbol x}$ in the $L^{1}$-norm, i.e.,
\[ \| f^{h}_{\boldsymbol x} - \bar f^{h}_{\boldsymbol x} \|_{L^{1}}
<\varphi(h),\quad \forall \boldsymbol x\in\mathbb R^k.\]
\end{itemize}
Also, we require that the two coupling processes $( \mathcal{X}_{t}, \mathcal{Y}_{t})$ and $( \mathcal{\bar{X}}_{t},
\mathcal{\bar{Y}}_{t})$ 
use the same Brownian motion as the true SDE
trajectory to produce the discrete random variables. This makes the comparison  between the numerical
and true SDE trajectories possible. 

Essentially,   {\bf (S1)}--{\bf (S4)}  assume that (i) $\bar{\bm X}^h$ is a
strong approximation of ${\bm X}^h$ with a good control of the tails;  (ii) the probability density function of $X^{h}_{1}$
given $X^{h}_{0}$ ({\it resp.} $\bar{X}^{h}_{1}$
given $\bar{X}^{h}_{0}$) is a good approximation of a Gaussian function
with variance $O(h).$  
 We  justify the assumptions {\bf (S1)}--{\bf (S4)} by numerical computations or the theoretical arguments. Please see the following Remark \ref{rem}.
\begin{rem}\label{rem}
Assumptions {\bf (S2)} and {\bf (S4)} are both about the transition
probability density function. When the Euler-Maruyama scheme is used,
$\bar{p}_{\boldsymbol z}$ and $\bar{f}_{\boldsymbol x}^{h}$ are probability density functions
of normal distributions, and for $h$  sufficiently small, $p_{\boldsymbol z}$ and $f^h_{\boldsymbol x}$ are also closely approximated by the normal probability density functions. Hence, {\bf (S2)} and {\bf (S4)} are reasonable assumptions. In particularly,
 if $g$ and $\sigma$ in
\eqref{SDE} are constants, {\bf (S2)} holds easily, as well as the first inequality 
in {\bf (S4)}.
For  the second inequality in {\bf (S4)}, it is not easy to integrate $|\bar{f}^{h}_{\boldsymbol x} - \bar{f}^{h}_{\boldsymbol x + h^{\gamma} \boldsymbol{v}}|$ by
hand, which is equivalent to integrating
\[
  \int_{\mathbb{R}^{d}} (2 \pi)^{-d/2} \mathrm{det}(h \Sigma)^{-1/2}
  |e^{-\boldsymbol{x}^{T}(h\Sigma)^{-1} \boldsymbol{x} /2} 
 -   e^{-(\boldsymbol{x} - h^{\gamma}\boldsymbol{v})^{T}(h\Sigma )^{-1}
   (\boldsymbol{x} - h^{\gamma} \boldsymbol{v})/2 }| \mathrm{d}\boldsymbol{x},
\]
where $\Sigma = \sigma(\boldsymbol x)$. We numerically compute the $L^{1}$ distance
between the probability density function of $N(0, h)$ and  $N(h^{0.8}, h)$
for different small $h$, and plotted in Figure \ref{consistent}  Left. The $L^{1}$ distance converges to zero at
a power-law speed with respect to $h$.  This verifies {\bf (S4)}.

For {\bf (S1)}, we
numerically compare the strong error of Milstein scheme for the
geometric Brownian motion 
$\mathrm{d}X_{t} = 0.2 X_{t} \mathrm{d}t +
X_{t}\mathrm{d}W_{t}$ at $t = 1$. The ratio
$\mathbb{P}[|\bar{X}_{1} - X_{1}| > h^{0.55}]/h^{1.1}$ for different
$h$ is plotted in Figure \ref{consistent} Left. We see that the
ratio decreases with respect to $h$. Hence {\bf (S1)} is satisfied. 

In practice, the  threshold $d$ in {\bf (S3)} can be set as $d=C\sqrt{h}$ with $C$ being in the same scale as $\epsilon$ (here, we choose $C=2\epsilon$). In this way, the two trajectories can be coupled with certain reasonable  probability once the maximal coupling is triggered. Our
numerical study finds that the numerical coupling time is not very sensitive against $C$.
\begin{figure}[H]
\centerline{\includegraphics[width = 0.75\linewidth]{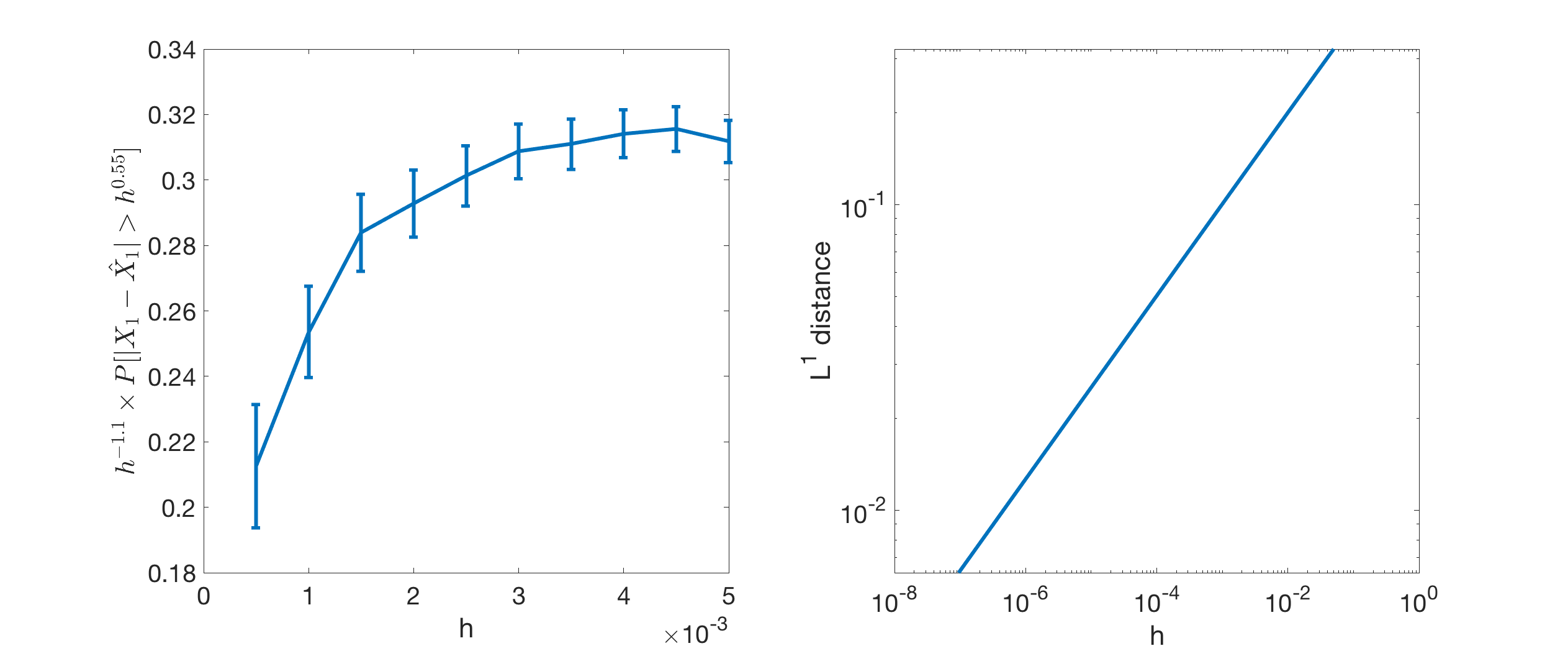}}
\caption{Left: $\mathbb{P}[|\bar{X}_{1} - X_{1}| > h^{0.55}]/h^{1.1}$ for different
$h,$ where $X_{t}$ is a geometric Brownian
motion. Error bars are included. Right: $L^{1}$ distance between the probability density function of $N(0, h)$ and that of $N(h^{0.8}, h)$ for different small $h$. }
\label{consistent}
\end{figure}

\end{rem}

\begin{thm}\label{compare}
Let $\tau_{c}$ and $\bar{\tau}_{c}$ be the coupling times of
$({\bm X}, {\bm Y})$ and $(\bar{\bm X}, \bar{\bm Y})$, respectively. Assume that {\bf (S1)}-- {\bf (S4)} hold. Then for any finite $t>0,$ there exists $a(t)>0$ such that for any $b\in(0,p-\frac{1}{2})$ and any $h>0$ sufficiently small, it holds that 
\begin{eqnarray}\label{difference}
\big|\mathbb{P}[ \tau_c > t]-
\mathbb{P}[\bar \tau_c > t]\big|\le a(t)h^\alpha+c\varphi(h)+
h^{p-\frac{1}{2} -b},
\end{eqnarray}
where $c>0$ is a uniform constant for all small $h>0,$ and the parameters $\alpha,p$ are as in {\bf (S1)}.
In particularly, \[\lim_{h\to0}\mathbb{P}[ \bar\tau_c > t]=
\mathbb{P}[\tau_c > t].\] 
\end{thm}

By Theorem \ref{compare}, if an extrapolation of small $h$ shows that the exponential tail of the
coupling time of $(\bar{\bm X}^h, \bar{\bm Y}^h)$ is strictly away from zero, then  the numerical coupling
provides a lower bound of the geometric convergence/contraction rate of
the SDE \eqref{SDE}. 

In the remainder of this section we prove 
Theorem \ref{compare}. 
Before proceeding to the proof, we briefly describe the idea of it. 
Observe that at each step $i,$ if the coupling succeeds for $(\mathcal X^h_{i}, \mathcal Y^h_{i})$,  then at the previous step $(i-1)$, $\mathcal X^h_{i-1}$ and $\mathcal Y^h_{i-1}$ must be sufficiently close so that the maximal coupling is triggered.
The strong approximation property {\bf (S1)} then  guarantees that at the step $(i-1),$ very likely, the maximal coupling is also  triggered  for the numerical coupling $(\bar{\mathcal X}^h_{n}, \bar{\mathcal Y}^h_{n})$. 
 In other words, the maximal coupling is triggered for one coupling process   while not for the other can only happen  with small probability.  
The events $ \mathrsfs{A}_{i-1},  \mathrsfs{\bar A}_{i-1},  \mathrsfs{B}_{i-1}$ and $ \mathrsfs{C}_{i-1}$ defined below 
as well as Proposition \ref{prop:events} are to indicate this. 
Moreover, whenever the maximal coupling is triggered, as long as $\mathcal X^h_{i-1}, \bar{\mathcal X}^h_{i-1}$ and $\mathcal Y^h_{i-1}, \bar{\mathcal Y}^h_{i-1}$ are both close, the probabilities to achieve a successful coupling at the next step are about the same. Lemma \ref{lem:estimate1}  is to establish this.  Although  there are situations when $|\mathcal X^h_n- \mathcal Y^h_n|$ ({\it resp.} $|\bar{\mathcal X}^h_n-\bar{\mathcal Y}^h_n|$) falls at the ``edge" of the triggering area, the probabilities are small as  stated by Lemma \ref{lem:estimate}.

\bigskip

\noindent{\it Proof of Theorem \ref{compare}:}
For convenience,  write $t = n_hh$ where $n_h=\lceil\frac{t}{h}\rceil.$ Recall that we use $n_c$ (resp. $\bar n_c$) to denote the numerical steps for a successful coupling 
for the time-$h$ chain ${\bm X^h}$ (resp. ${\bm{\bar X}^h}$), where $\tau_c=hn_c$ (resp. $\bar\tau_c=h\bar n_c$).
Then 
\begin{align*}
\big|\mathbb{P}[ \tau_c > t]-
\mathbb{P}[\bar \tau_c > t]\big|
= \big| \mathbb{P}[ n_{c} > n_h] - \mathbb{P}[\bar{n}_{c} > n_h] \big|
= \big|\mathbb{P}[ n_{c} \leq n_h] - \mathbb{P}[\bar{n}_{c} \leq n_h]\big |.\nonumber
\end{align*}
Since the maximal coupling is triggered no early than the second step,  we have
\begin{eqnarray*}
\big|
\mathbb{P}[ n_{c} \leq n_h] - \mathbb{P}[\bar{n}_{c} \leq n_h]\big ||\le\sum_{i=2}^{n_h}\big||\mathbb{P}[ n_{c} = i] - \mathbb{P}[ \bar{n}_{c} = i]\big|.
\end{eqnarray*}

As discussed above,  for each $2\le i\le n_h,$ if the coupling between  $\mathcal{\bar{X}}^{h}_{n}$ and $\mathcal{\bar{Y}}^{h}_{n}$ occurs at step $i$,
then at  step $(i-1),$  besides that the maximal coupling of $(\mathcal X^h_n, \mathcal Y^h_n)$ must be  triggered, 
the maximal coupling of 
$(\bar{\mathcal X}^h_n, \bar{\mathcal Y}^h_n)$ is  (very likely) triggered as well. 
To clarify this, we  split each term  $\Big(\mathbb{P}[ n_{c} = i] - \mathbb{P}[ \bar{n}_{c} = i]\Big)$  according to whether the coupling process at the step $(i-1)$ falls at the  ``edge'' of the triggering area. 
The following several events are defined according to this.
Fix $\delta \in(0, \frac{p-\frac{1}{2}}{3}).$ Let 
\[
\mathrsfs A_{i-1} = \Big\{ |\mathcal{X}^{h}_{i-1} - \mathcal{Y}^{h}_{i-1}| < d - h^{p -
	\delta},n_c>i-1 \Big\},
\]
\[
\mathrsfs {\bar{A}}_{i-1} = \Big\{ |\mathcal{\bar{X}}^{h}_{i-1} - \mathcal{\bar{Y}}^{h}_{i-1}| < d - h^{p -
	\delta},\bar n_c>i-1  \Big\},
\]
\[
 \mathrsfs B_{i-1}= \Big\{| \mathcal{\bar{X}}^{h}_{i-1} -
\mathcal{X}^{h}_{i-1}| \leq \frac{1}{2}h^{p - \delta}\,, | \mathcal{\bar{Y}}^{h}_{i-1} -
\mathcal{Y}^{h}_{i-1}| \leq \frac{1}{2}h^{p - \delta}, n_c >i-1, \bar n_c>i-1\Big\},
\]
\begin{align*}
 \mathrsfs C_{i-1}&=\Big\{ |\mathcal{X}^{h}_{i-1} - \mathcal{Y}^{h}_{i-1}| < d  \, ,
|\mathcal{\bar{X}}^{h}_{i-1} - \mathcal{\bar{Y}}^{h}_{i-1}| < d, 
| \mathcal{\bar{X}}^{h}_{i-1} -
\mathcal{X}^{h}_{i-1}| \leq \frac{1}{2}h^{p - \delta}\,,\\ &| \mathcal{\bar{Y}}^{h}_{i-1} -
\mathcal{Y}^{h}_{i-1}| \leq \frac{1}{2}h^{p - \delta},n_c>i-1,\bar n_c>i-1  \Big\},
\end{align*} 
where $p>1/2,d=2\epsilon\sqrt{h}$ are  from {\bf (S1)} and {\bf (S3)}, respectively. 
Note that 
the occurrence of both  events $ \mathrsfs A_{i-1}$ (resp.  $\mathrsfs{\bar A}_{i-1}$) and $ \mathrsfs B_{i-1}$ induces  the occurrence of the event $\mathrsfs C_{i-1}$, i.e., 
\[\mathrsfs A_{i-1}\backslash \mathrsfs C_{i-1}\subseteq \mathrsfs B^c_{i-1}\quad (resp.\quad \mathrsfs {\bar A}_{i-1}\backslash\mathrsfs   C_{i-1}\subseteq \mathrsfs B^c_{i-1}).\] 
Combined with the strong approximation property {\bf (S1)}, we immediately obtain  the following estimates.
\begin{prop}\label{prop:events}
For each $2\le i\le n_h,$ it holds that
\begin{eqnarray*}
\mathbb{P}[n_{c} = i, \mathrsfs A_{i-1} \setminus \mathrsfs C_{i-1}] \leq 2C(t)h^{1 +{\alpha}}\ (resp. \quad \mathbb{P}[\bar{n}_{c} = i, \mathrsfs{\bar{A}}_{i-1} \setminus \mathrsfs C_{i-1}]\leq  2C(t)h^{1 +{\alpha}}),
\end{eqnarray*}
where $C(t)$ is from {\bf (S1)}.
\end{prop}

\medskip

Now, for each $2\le i\le n_h,$ we split $\mathbb P[n_c=i]$ (resp. $\mathbb P[\bar n_c=i]$) as 
\begin{eqnarray*}
\mathbb{P}[n_{c}=i]
=\mathbb P [n_c=i, \mathrsfs C_{i-1}]+\mathbb{P}[n_{c} = i, \mathrsfs A_{i-1} \setminus\mathrsfs C_{i-1}]+\mathbb P[n_c=i, \mathrsfs A_{i-1}^c].
\end{eqnarray*}
(resp.  $\quad \mathbb{P}[\bar{n}_{c}=i] =\mathbb P [\bar{n}_c=i, \mathrsfs{\bar A}_{i-1}]+\mathbb{P}[\bar n_{c} = i, \mathrsfs{\bar A}_{i-1} \setminus\mathrsfs  C_{i-1}]+\mathbb P[\bar{n}_c=i, \mathrsfs{\bar A}_{i-1}^c]$).

By Proposition \ref{prop:events} we have 
\begin{eqnarray*}
|\mathbb{P}[ n_{c} = i] - \mathbb{P}[ \bar{n}_{c} = i]|
&\leq& | \mathbb{P}[ n_{c} = i , \mathrsfs C_{i-1}] - \mathbb{P}[ \bar{n}_{c} =
i, \mathrsfs C_{i-1}] | + 4C(t)h^{1+\alpha}\\
&+&|\mathbb P[n_c=i, \mathrsfs A_{i-1}^c]-\mathbb P[\bar{n}_c=i, \mathrsfs{\bar A}_{i-1}^c]|.
\end{eqnarray*}
Hence, the estimation of $|\mathbb{P}[ n_{c} = i] - \mathbb{P}[ \bar{n}_{c} = i]|$ 
is reduced to the estimations of 
\begin{eqnarray*}
| \mathbb{P}[ n_{c} = i ,\mathrsfs C_{i-1}] - \mathbb{P}[ \bar{n}_{c} =
i, \mathrsfs C_{i-1}] | 
\end{eqnarray*}
and 
\begin{eqnarray*}\label{sum2}
|\mathbb P[n_c=i, \mathrsfs A_{i-1}^c]-\mathbb P[\bar{n}_c=i, \mathrsfs{\bar A}_{i-1}^c]|.
\end{eqnarray*}
These are concluded  by  the following two lemmata.

\begin{lem}\label{lem:estimate1}
For each $2\le i\le n_h,$  it holds that 
\begin{eqnarray*}
|\mathbb P [n_c=i, \mathrsfs C_{i-1}]-\mathbb P [\bar{n}_c=i, \mathrsfs C_{i-1}]|\le c_0\varphi(h)\mathbb{P}[ n_{c} = i],
\end{eqnarray*}
 where $c_0>0$ is a uniform constant for all $i$ and small $h>0$.
\end{lem}

\begin{lem}\label{lem:estimate}
For each $2\le i\le n_h,$  the following  hold 
\begin{eqnarray}
\mathbb P[n_c=i, \mathrsfs A_{i-1}^c]&\le& c_1h^{p -
\frac{1}{2}  - 2\delta}\mathbb{P}[ n_{c} = i]\label{1}\\
\mathbb P[\bar{n}_c=i, \mathrsfs{\bar A}_{i-1}^c]&\le& c_1h^{p -
		\frac{1}{2}  - 2\delta}\mathbb{P}[ \bar{n}_{c} = i]\label{2},
\end{eqnarray}
where $c_1>0$ is a uniform constant for all $i$ and small $h>0$.
\end{lem}

\medskip

We postpone the proofs of Lemma \ref{lem:estimate1} and Lemma \ref{lem:estimate} to the end.
Combining all the estimates above, 
\begin{eqnarray*}
	& & | \mathbb{P}[n_{c} = i] - \mathbb{P}[ \bar{n}_{c} = i] | \\
	&\le&
	4C(t)h^{1+\alpha} + c_0\varphi(h) \mathbb{P}[n_{c} = i] + c_1h^{p
		- \frac{1}{2} - 2\delta}(\mathbb{P}[n_{c} = i]  +
	\mathbb{P}[\bar{n}_{c} = i]).
\end{eqnarray*}
Note that
$$
\sum_{i = 2}^{n_h}\mathbb{P}[n_{c} = i] \leq 1, \quad
\sum_{i = 2}^{n_h}\mathbb{P}[\bar{n}_{c} = i] \leq 1.
$$
Then together with $n_h$ being in the order $O(t/h),$
we finally obtain 
\begin{eqnarray*}
& &\sum_{i = 2}^{n_h} | \mathbb{P}[n_{c} = i] -
	\mathbb{P}[ \bar{n}_{c} = i] |\\
	&\le&
	4n_hC(t)h^{1+\alpha} + c_0\varphi(h)\sum_{i = 2}^{n_h} \mathbb{P}[n_{c} = i] + c_1h^{p
		- \frac{1}{2} - 2\delta}\sum_{i = 2}^{n_h}(\mathbb{P}[n_{c} = i]  +
	\mathbb{P}[\bar{n}_{c} = i] )\\
	&\le& \tilde C(t)h^\alpha+c_0\varphi(h)+2c_1
	h^{p-\frac{1}{2} - 2\delta},
 \end{eqnarray*}
where $\tilde C(t)>0$ only depend on $t.$ 

Now, Theorem \ref{compare} is proved by setting $a(t)=\tilde C(t)$ and $b=3\de.$

\bigskip

\noindent{\it Proof of Lemma \ref{lem:estimate1}:}
Since $(p-\delta)>1/2,$ if denote $f_{\boldsymbol x}, f_{\bar{\boldsymbol x}}$ ({\it resp.} $f_{\boldsymbol y}, \bar f_{\bar{\boldsymbol y} }$) as  the probability density functions of $\mathcal X^h_{i}, \bar{\mathcal X}^h_{i}$ ({\it resp.} $\mathcal Y^h_{i}, \bar{\mathcal Y}^h_{i}$) conditioning on $\mathcal X^h_{i-1}=\boldsymbol x, \bar{\mathcal X}^h_{i-1}=\bar{\boldsymbol x}$ (resp. $\mathcal Y^h_{i-1}=\boldsymbol y, \bar{\mathcal Y}^h_{i-1}=\bar{\boldsymbol y}$),  by  {\bf (S4)}, we have
\begin{eqnarray*}
	\|f_{\boldsymbol x}-\bar  f_{\bar{\boldsymbol x}}\|_{L^1}\le 2\varphi(h)\ (resp.\ \|f_{\boldsymbol y}-\bar  f_{\bar{\boldsymbol y}}\|_{L^1}\le 2\varphi(h)).
\end{eqnarray*}
Then the mechanism of the maximal coupling yields
\begin{eqnarray*}
	&\ &| \mathbb{P}[ n_{c} = i, \mathrsfs C_{i-1}] - \mathbb{P}[ \bar{n}_{c}=
	i, \mathrsfs C_{i-1}] |\\
	&=&| \mathbb{P}[ n_{c}
	= i | \mathrsfs C_{i-1}] - \mathbb{P}[ \bar{n}_{c} =
	i | \mathrsfs C_{i-1}] |\cdot \mathbb{P}[\mathrsfs C_{i-1}] \leq 4\varphi(h) \mathbb{P}[\mathrsfs C_{i-1}]. 
\end{eqnarray*}

 Note that as long as the maximal coupling is triggered, the coupling probability is in order $O(1)$ and uniform with respect to all small $h>0,$ i.e., 
\begin{eqnarray*}
\mathbb P[n_c=i, \mathrsfs C_{i-1}]\ge \eta_0\mathbb P[\mathrsfs C_{i-1}]
\end{eqnarray*}
for a constant $\eta_0>0.$
Therefore, 
\[
| \mathbb{P}[ n_{c} = i, \mathrsfs C_{i-1}] - \mathbb{P}[ \bar{n}_{c} =
i, \mathrsfs C_{i-1}] | \leq (4\varphi(h)/\eta_0)\mathbb{P}[ n_{c} = i , \mathrsfs C_{i-1}]\leq
(4\varphi(h)/\eta_0)\mathbb{P}[ n_{c} = i].
\]

 Lemma \ref{lem:estimate1} is proved by letting $c_0=4/\eta_0.$

\bigskip

\noindent{\it Proof of Lemma \ref{lem:estimate}:}
We only need to prove \eqref{1}, and \eqref{2} can be obtained  similarly. 
First, we estimate $\mathbb{P}[ d - h^{p - \delta} \leq |
 \mathcal{X}^{h}_{i-1} - \mathcal{Y}^{h}_{i-1} | \leq d].$
Conditioning on the value at the step $(i-2)$, we have 
\begin{align*}
  &\mathbb{P}[ d - h^{p - \delta} \leq |
  \mathcal{X}^{h}_{i-1} - \mathcal{Y}^{h}_{i-1} | \leq d] \\
=& \int_{\mathbb R^k \times \mathbb R^k} \mathbb{P}[d - h^{p - \delta} \leq |
   \mathcal{X}^{h}_{i-1} - \mathcal{Y}^{h}_{i-1} | \leq d \,|\,
    \mathcal{X}^{h}_{i-2}=\boldsymbol x ,  \mathcal{Y}^{h}_{i-2}=\boldsymbol y]
   \mu_{i-2}(d \boldsymbol x,d\boldsymbol  y),
\end{align*}
where $\mu_{i-2}(d\boldsymbol x,d\boldsymbol y) $ is the joint probability distribution of $(
\mathcal{X}^{h}_{i-2}, \mathcal{Y}^{h}_{i-2}) $.

By {\bf (S2)}, the
probability density function of $(\mathcal{X}^{h}_{i-1} -\mathcal{Y}^{h}_{i-1})$ conditional on $  \mathcal{X}^{h}_{i-2}=\boldsymbol x ,  \mathcal{Y}^{h}_{i-2}=\boldsymbol y$ is Gaussian-like.
So we have the following
   comparison of $\mathbb{P}[ d - h^{p - \delta} \leq |
   \mathcal{X}^{h}_{i-1} - \mathcal{Y}^{h}_{i-1} | \leq d \,|\,
   \mathcal{X}^{h}_{i-2} =\boldsymbol x,  \mathcal{Y}^{h}_{i-2} =\boldsymbol y]$ and $\mathbb{P}[ |
   \mathcal{X}^{h}_{i-1} - \mathcal{Y}^{h}_{i-1} | \leq d \,|\,
   \mathcal{X}^{h}_{i-2} =\boldsymbol x,  \mathcal{Y}^{h}_{i-2} = \boldsymbol y]$ as follows:

  (i) If $|\boldsymbol x -\boldsymbol  y| \le -\delta \log h\cdot h^{1/2}.$ Since $d =O(h^{1/2})$, within the set $\{ |\mathcal{X}^{h}_{i-1} -
   \mathcal{Y}^{h}_{i-1} | \leq d  \}$,   the maximal density of
   $(\mathcal{X}^{h}_{i-1} - \mathcal{Y}^{h}_{i-1})$ is at most
   $O(h^{-\delta})$ times the minimal density of
   $(\mathcal{X}^{h}_{i-1} - \mathcal{Y}^{h}_{i-1})$. In consideration that the volume of the shell $\{(\boldsymbol u,\boldsymbol w)\in\mathbb R^k: d - h^{p - \delta} \leq |\boldsymbol u-\boldsymbol w| \leq d  \}$ is
   $O(h^{(k-1)/2 + p - \delta}),$  
   we can find  a constant $c>0$ such that 
\begin{align*}
  &\mathbb{P}[ d - h^{p - \delta} \le |
   \mathcal{X}^{h}_{i-1} - \mathcal{Y}^{h}_{i-1} | \leq d \,|\,
   \mathcal{X}^{h}_{i-2} =\boldsymbol  x,  \mathcal{Y}^{h}_{i-2} = \boldsymbol  y] \\
\leq &c
   h^{k/2+p - \frac{1}{2} - 2\delta} \mathbb{P}[ |
   \mathcal{X}^{h}_{i-1} - \mathcal{Y}^{h}_{i-1} | \leq d \,|\,
   \mathcal{X}^{h}_{i-2} =\boldsymbol  x,  \mathcal{Y}^{h}_{i-2} = \boldsymbol y];
\end{align*}

(ii) If $|\boldsymbol x - \boldsymbol y| > -\delta \log h \cdot h^{1/2}$. Then the probability density of $(\mathcal{X}^{h}_{i-1} - \mathcal{Y}^{h}_{i-1})$ within the set  $\{ |\mathcal{X}^{h}_{i-1} -
\mathcal{Y}^{h}_{i-1} | \leq d  \}$ is less
   than  $ce^{-(\delta\log h)^{2}}h^{-k/2}$ (here, we still use $c>0$ as a uniform constant),  which converges to zero
   faster than $h^{r}$ for any $r>0$. 
   Hence,
\begin{align*}
   &\mathbb{P}[ d - h^{p - \delta} \le |
   \mathcal{X}^{h}_{i-1} - \mathcal{Y}^{h}_{i-1} | \leq d \,|\,
   \mathcal{X}^{h}_{i-2} =\boldsymbol x,  \mathcal{Y}^{h}_{i-2} = \boldsymbol y] \\
   \leq & 
   h^{p - \frac{1}{2} - \delta+r} \mathbb{P}[ |
   \mathcal{X}^{h}_{i-1} - \mathcal{Y}^{h}_{i-1} | \leq d \,|\,
   \mathcal{X}^{h}_{i-2} =\boldsymbol x,  \mathcal{Y}^{h}_{i-2} =\boldsymbol  y].
 \end{align*}
 Now, for both  cases, integrating over the initial conditions $(\boldsymbol x,\boldsymbol y)$, we have
\begin{eqnarray*}
&\mathbb{P}[d - h^{p - \delta} \leq |
\mathcal{X}^{h}_{i-1} - \mathcal{Y}^{h}_{i-1} | \leq d]
\le h^{p -
	\frac{1}{2}  - 2\delta}\mathbb{P}[|
\mathcal{X}^{h}_{i-1} - \mathcal{Y}^{h}_{i-1} | \leq d]. 
\end{eqnarray*}

Consequently,
\begin{align*}
  &\mathbb{P}[ n_{c} = i, \mathrsfs A^c_{i-1}] \\
 =&\mathbb{P}[ n_{c} = i| d - h^{p - \delta} \leq |
 \mathcal{X}^{h}_{i-1} - \mathcal{Y}^{h}_{i-1} | \leq d]\cdot\mathbb P[d - h^{p - \delta} \leq |
 \mathcal{X}^{h}_{i-1} - \mathcal{Y}^{h}_{i-1} | \leq d]\\
 \leq& h^{p -
 	\frac{1}{2}  - 2\delta}\cdot\mathbb{P}[ n_{c} = i|
 d-h^{p-\de}\le| \mathcal{X}^{h}_{i-1} - \mathcal{Y}^{h}_{i-1} | \leq d]\cdot \mathbb P[
 |\mathcal{X}^{h}_{i-1} - \mathcal{Y}^{h}_{i-1} | \leq d]
\end{align*}
As in the proof of Lemma \ref{lem:estimate1}, since the coupling probability conditioning on the event 
$\{|\mathcal{X}^{h}_{i-1} - \mathcal{Y}^{h}_{i-1} | \leq d\}$ is  uniform for all small $h>0,$ we have
\begin{eqnarray*}
\mathbb{P}[ n_{c} = i|
d-h^{p-\de}\le| \mathcal{X}^{h}_{i-1} - \mathcal{Y}^{h}_{i-1} | \leq d]\le \mathbb{P}[ n_{c} = i|\mathcal{X}^{h}_{i-1} - \mathcal{Y}^{h}_{i-1} | \leq d]/\eta_0,
\end{eqnarray*} 
where $\eta_0$ is as in the proof of Lemma \ref{lem:estimate1}. 
Thus, 
\[\mathbb{P}[ n_{c} = i, \mathrsfs A^c_{i-1}] \le (h^{p -
	\frac{1}{2}  - 2\delta}/\eta_0)\mathbb{P}[ n_{c} = i].\]
By setting $c_1=1/\eta_0$,  Lemma \ref{lem:estimate} is proved.

\subsection{Overdamped Langevin dynamics}
The first SDE example we shall use is the
overdamped Langevin dynamics. Consider 
\begin{equation}
\label{gradient} 
  \mathrm{d}X_{t} = - \nabla V(X_{t}) + \epsilon \mathrm{d}W_{t},
\end{equation}
where $V(x)$ is a potential function. It is well known that \eqref{gradient} admits a unique invariant probability measure
$\pi_{\epsilon}$ with the probability density
$$
  \rho_{\epsilon} = \frac{1}{K}e^{-2V(X)/\epsilon^{2}},
$$
where $K$ is a normalizer. In addition,  if $V$ is strictly convex such that
$\mathrm{Hess}(V) - R \mathrm{Id}_{k}$ is positive definite, then
$\pi_{\epsilon}$ satisfies the Logarithmic Sobolev inequality with
constant $\epsilon^{2} R/2$. Hence, the geometric convergence rate is
at least $R$. (We refer to \cite{bakry1985diffusions,lelievre2016partial} for details.) Now we
check our numerical result for the rate of geometric ergodicity with the above
analytical result. 

Consider $n = 2$ and $V(x,y) = (x^{2} + y^{2})/2$. This potential
function is strictly convex with Hessian matrix $\mathrm{Id}_{2}$. We run {\bf
  Algorithm 2} for different steps sizes $h = 0.0005, 0.001,
0.0015, 0.002, 0.0025$ and $0.003$. Throughout this section, the threshold of triggering the maximal coupling is set as $d = 2 \epsilon \sqrt{h}$ . The sample size $N =
10^{7}$. To reach the optimal coupling rate, we use the reflection coupling until the maximal coupling is triggered. Coupling time distributions versus different 
step sizes are compared in a log-linear plot (Figure \ref{fig5}
Left). The slopes of those exponential tails are computed by fitting
$\log \mathbb{P}[\tau_{c} > t]$ versus $t$ using a linear
function. We linearly extrapolate the negative slopes for decreasing $h$ in Figure \ref{fig5}
Right. We see that the numerical result for the rate of geometric ergodicity is very close to the theoretical one. In
addition, a smaller time step size gives a higher rate. By Theorem \ref{compare}, these numerically computed rates of
geometric ergodicity are trustable. 

\begin{figure}[htbp]
\centerline{\includegraphics[width = \linewidth]{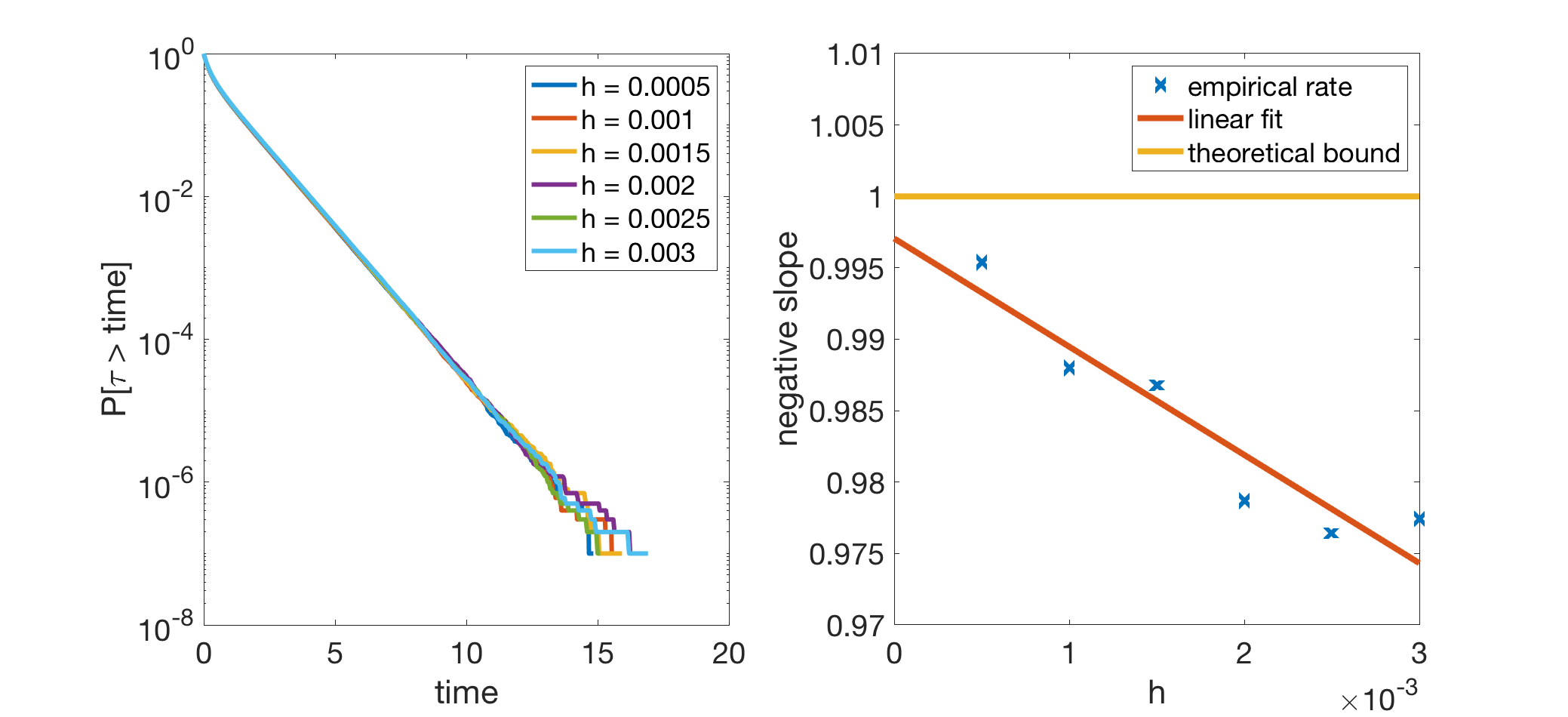}}
\caption{ Left: Coupling time distribution of the overdamped Langevin
  dynamics under different time step sizes. Right: Comparison of
  exponential tails of coupling time with different time step sizes.}
\label{fig5}
\end{figure}

\subsection{Van der Pol oscillator}
The second SDE example is the Van der Pol oscillator with additive
noise. We use this example to demonstrate the effect of slow-fast
dynamics on the geometric ergodicity. Consider
\begin{align}
  \label{VDP}
\mathrm{d}X_{t} &= (X_{t} - \frac{1}{3}X_{t}^{3} - Y_{t}) \mathrm{d}t +
  \epsilon \mathrm{d}W^{1}_{t}\\\nonumber
\mathrm{d}Y_{t} &= \frac{1}{\mu} X_{t} \mathrm{d}t + \epsilon \mathrm{d}W^{2}_{t}
\end{align}
The deterministic part of  \eqref{VDP} admits a limit cycle,
as shown in Figure \ref{fig6} Top Left. When $\mu \gg 1$, this system
demonstrates the slow-fast dynamics, which is called the
relaxation oscillation. The solution will move slowly along left/right
side of the limit cycle for a long time, and then jump to the other side
quickly after passing the ``folding point''. See Figure
\ref{fig6} Top Middle for $x$-trajectory versus time of the
deterministic equation.

The Van der Pol oscillator has been studied
for decades. We shall use our coupling methods to numerically study the
spectral property of \eqref{VDP}. The magnitude of noise is
chosen as $\epsilon = 0.3$, which is small compared with the size
of the limit cycle. We run {\bf Algorithm 2} with $N =
10^{7}$ samples and time step size is set as $h = 0.001$. Before the two trajectories are sufficiently close to each other, we
use a mixture of the independent and  reflection couplings. More
precisely, at each step,  with probability $\beta$ we use the
independent coupling,  and use the reflection coupling for
otherwise. This makes
the coupling process irreducible.  In the first simulation,  we fix $\mu
= 12$ and let $\beta = 0, 0.02, 0.04, 0.06,
0.08, 0.1$.  We find that the resultant
rate of the exponential tails decreases slightly as $\beta$ increases since  the
reflection coupling is more efficient than the independent coupling. However, this dependency is not very
sensitive; see Figure \ref{fig6} Top Right and Middle Left for more details. 

In the second simulation, we fix $\beta = 0.05$ and let $\mu = 2, 4, 6, 8, 10,
12$. The exponential tails of the  coupling time distribution  corresponding to the different $\mu$'s are
compared; see Figure \ref{fig6} Middle Right and Bottom Left. Note that the Middle Right figure is cut off at the probability $10^{-5}$ and horizontally stretched in order to demonstrate the
difference between $\mu = 10$ and $\mu = 12$ plots. The slopes of
these exponential tails versus different $\mu$'s are computed and plotted in Figure
\ref{fig6} Bottom Left.   

\begin{figure}[h]
\centerline{\includegraphics[width = 1.2\linewidth]{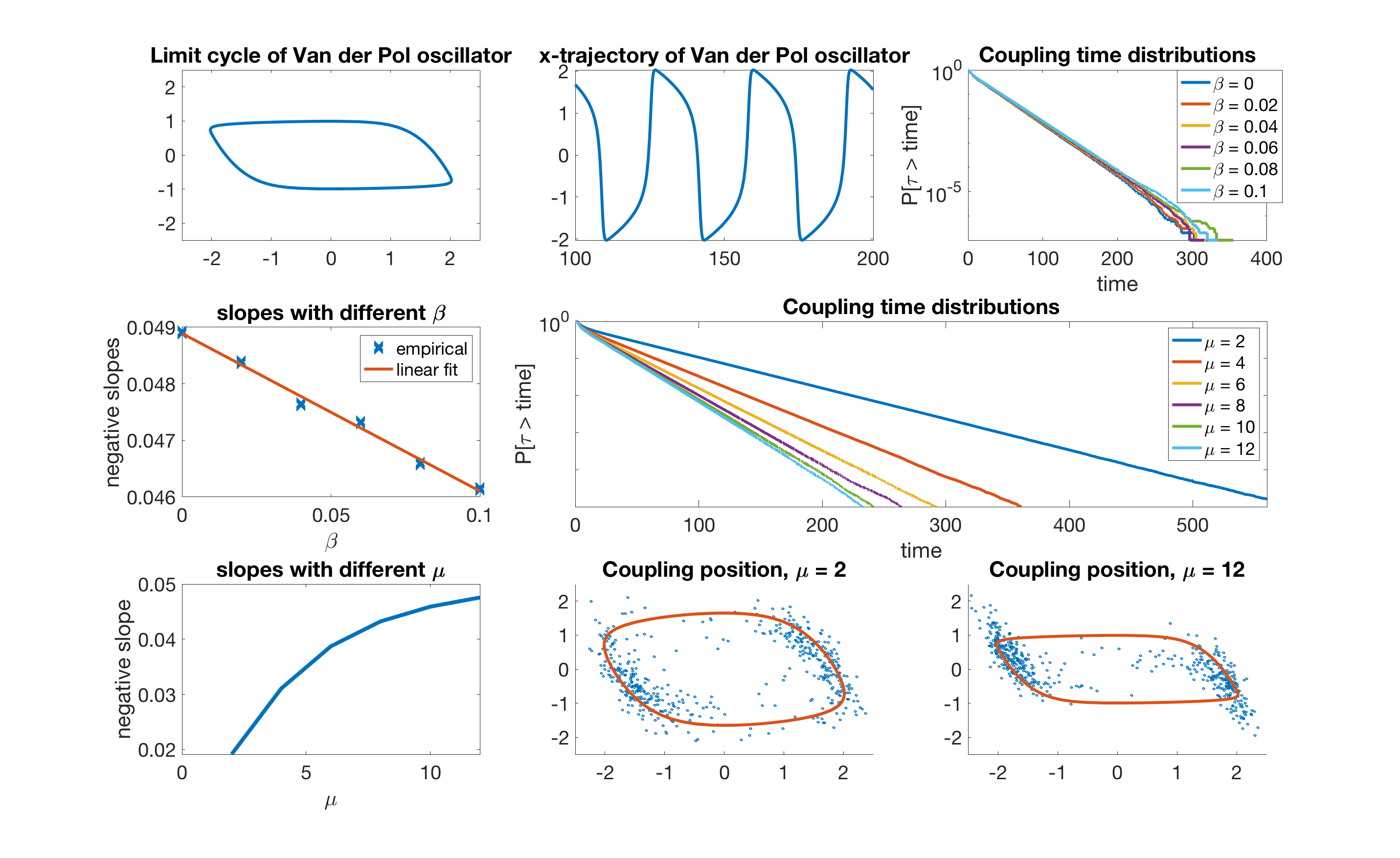}}
\caption{Top Left: limit cycle of the Van der Pol oscillator for $\mu
  = 12$. Top
  Mid: Deterministic trajectory of $x$-variable. Top Right:
  $\mathbb{P}[\tau_{c} > t]$ versus $t$ for different
  values of $\beta$ in log-linear plot. Mid Left: Linear fit of negative slopes of
  $\mathbb{P}[\tau_{c} > t]$ versus $t$ for different values of
  $\beta$. Mid Right:
  $\mathbb{P}[\tau_{c} > t]$ versus $t$ for different
  values of $\mu$  in log-linear plot. Bottom Left: Negative slopes of
  $\mathbb{P}[\tau_{c} > t]$ versus $t$ for different values of
  $\beta$. Bottom Mid: Positions where two trajectories couple when
  $\mu = 2$. Bottom Right: Positions where two trajectories couple when
  $\mu = 12$ }
\label{fig6}
\end{figure}

In this example, the rate of geometric ergodicity is small. This is expected because
one trajectory needs to diffuse along the limit cycle to ``chase'' the
other trajectory, which takes a considerable amount of time. An interesting observation is that the rate of geometric ergodicity
increases significantly with the increased time separation scale $\mu$. In other
words, a larger time-scaling separation of the slow-fast dynamics make the law of \eqref{VDP}
converge to its steady state distribution faster. To the best of our
knowledge, this interesting phenomenon is not documented in the previous
studies. We believe the reason is that a larger $\mu$ makes a trajectory move both slower near the slow manifold and closer to it, which
significantly increase the chance for two trajectories to ``meet''. This is confirmed numerically by Figure \ref{fig6} Bottom Middle and
Right. The positions of  $500$ samples  are plotted at which  they  are coupled
for $\mu = 2$ and $12$ respectively.  We see that the larger $\mu$ makes the
trajectories more likely to couple near the slow manifolds (the left
and right branches of the limit cycle and its extensions).

\subsection{SIR model with degenerate noise}
In this subsection, we  use an SIR model with degenerate noise to
demonstrate how our algorithm cam be adapted for SDEs with degenerate
diffusion terms. For degenerate diffusions, only one step of the numerical
algorithm does not produce a well-defined probability density function.  We need more than one step to implement the
maximal coupling. 

Consider an epidemic model in which the whole
population is divided into three distinct classes $S$ (susceptible
class), $I$ (infected class), and $R$ (recovered class), respectively. An SIR model
with the population growth is given by
\begin{align}\label{SIR}
\mathrm{d}S &= (\alpha - \beta SI - \mu S) \mathrm{d}t \\\nonumber
\mathrm{d}I &= (\beta SI - (\mu + \rho + \gamma) I ) \mathrm{d}t\\\nonumber
\mathrm{d}R &= (\gamma I - \mu R) \mathrm{d}t,
\end{align}
where $\alpha$ is the population birth rate, $\mu$ is the disease-free
death rate, $\rho$ is the excess death rate for the infected class,
$\gamma$ is the recover rate for the  infected population, and $\beta$ is
the effective contact rate between the susceptible class and infected
class \cite{dieu2016classification}. This model has been intensively
studied. We refer \cite{capasso1993mathematical, kermack1932contributions,  kermack1991contributions}
for a few representative references. 

Assume that all the three classes are driven by the same random
factor (such as temperature, humidity, etc.). This gives the SDE a  degenerate noise. Note that $S$ and $I$ in  \eqref{SIR} are independent of $R$. So we  consider the following SDE instead 
\begin{align}
  \label{SIrandom}
\mathrm{d}S &= (\alpha - \beta SI - \mu S) \mathrm{d}t + \sigma S \mathrm{d}W_{t}\\\nonumber
\mathrm{d}I &= (\beta SI - (\mu + \rho + \gamma) I ) \mathrm{d}t +
              \sigma I \mathrm{d}W_{t},
\end{align}
where $\sigma>0$ is the intensity of the white noise, and the two $\mathrm{d}W_{t}$ terms are from the same Brownian
motion. See Figure \ref{figSIR} Left for the trajectory in $\mathbb{R}^{2}_{+}$.

In \cite{dieu2016classification}, several results about the  asymptotic
behaviors of \eqref{SIrandom} are proved. Let 
$$
  \lambda = \frac{\alpha \beta}{\mu} - (\mu + \rho + \gamma -
  \frac{\sigma^{2}}{2}).
$$
If $\lambda > 0$, then  \eqref{SIrandom} admits a non-degenerate
invariant probability measure on $\mathbb{R}^{2}_{+}$. In
addition, it was shown that  
\eqref{SIrandom} approaches to its invariant probability measure faster than any
polynomial of $t$. This result is later improved in \cite{nguyen2020general}. In
this example, it is very challenging to construct an optimal Lyapunov
function to control the two different factors simultaneously. The Lyapunov function of
 \eqref{SIrandom} must take high values when $S$ and $I$ are
either too large or too small. A different approach is used in
\cite{nguyen2020general} to show the exponential ergodicity, but the resultant
rate is still not quantitative.

We use {\bf Algorithm 2} with an adaptation to the  degenerate noise (which
will be explained later) to examine the ergodicity of
\eqref{SIrandom}. The model parameters are set as $\alpha =7$, $\beta = 3$, $\mu = 1$, $\rho = 1$, $\gamma = 2$, and $\sigma =1$, the same as the example used in \cite{dieu2016classification}. Note that the reflection coupling cannot be applied due to the degeneracy of the noise.
In fact, for this set of parameters, 
the deterministic part of \eqref{SIrandom} converges to a unique equilibrium. With the same random noise being applied each time, any pair of stochastic trajectories of  \eqref{SIrandom} will converge to each other, just as its deterministic part does.  So in {\bf Algorithm 2},  we first use  the synchronous coupling to make  the two trajectories  sufficiently close. Then we implement a ``two-step version" of the maximal coupling to check whether the two trajectories can couple after every two steps. The numerical algorithm we use is still the Euler-Maruyama method  with the step size $h = 0.001$.  The total sample size is $N=10^{8}$. The coupling time distribution is demonstrated in Figure
  \ref{figSIR} Right. We can clearly see an exponential tail for
  $\mathbb{P}[ \tau_{c} > t]$. The linear fitting of $\log \mathbb{P}[
  \tau_{c} > t]$ versus $t$ gives a slope $\approx -
  0.53349$. Therefore, we conclude that \eqref{SIrandom} is indeed geometrically ergodic.

\begin{figure}[htbp]
\centerline{\includegraphics[width = \linewidth]{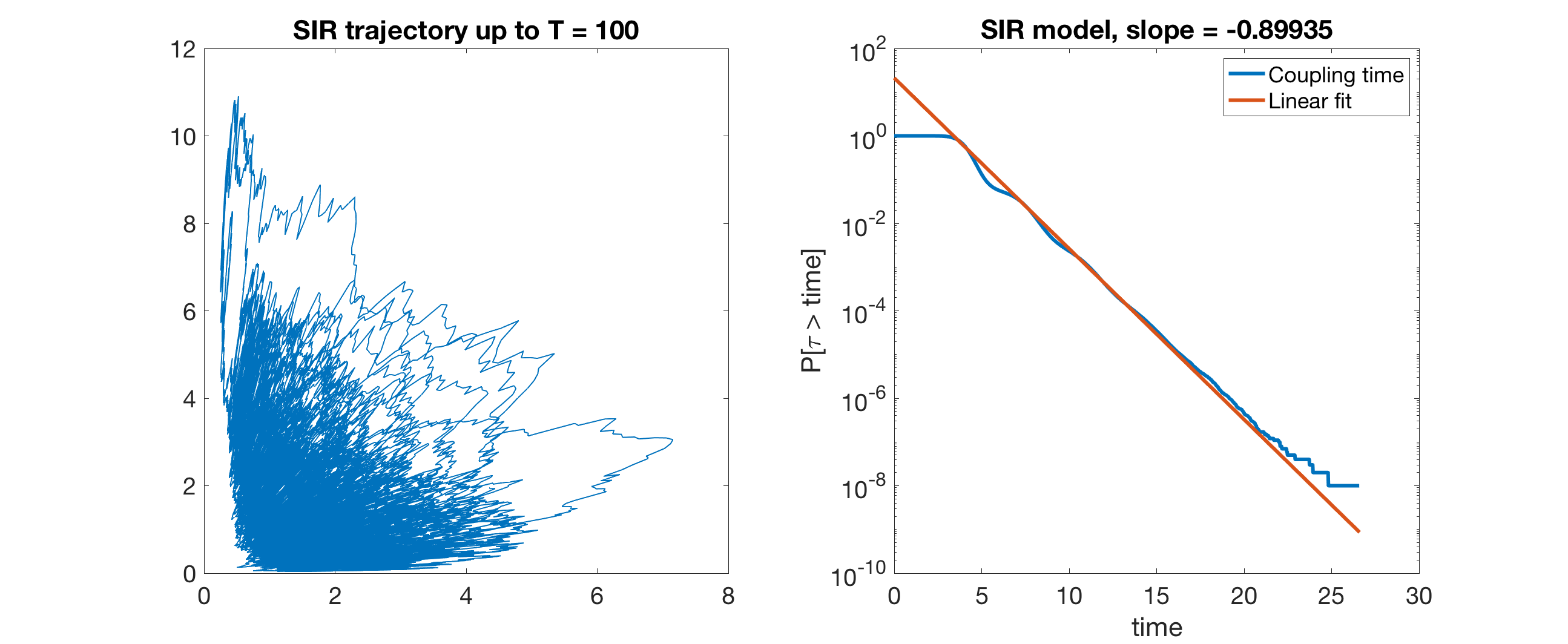}}
\caption{Left: Trajectory of equation \eqref{SIrandom} up to $T =
  100$. Right: Coupling time distribution $\mathbb{P}[ \tau_{c}
  > t]$ vs. $t$ in log-linear plot and linear function fitting. Parameters are $\alpha =
  7$, $\beta = 3$, $\mu = 1$, $\rho = 1$, $\gamma = 2$, and $\sigma =
  1$. }
\label{figSIR}
\end{figure}

\medskip

Now, we explain how to adapt {\bf Algorithm 3} for the degenerate diffusions. Since the one-step transition probability density function
of  \eqref{SIrandom} is degenerate, the density functions $p^{(x)}$ and
$p^{(y)}$ in {\bf Algorithm 3} are not well-defined. Instead, we need
to manually calculate the two-step transition probability density
function and then run the maximal coupling for two successive steps. Hence, the output
in {\bf Algorithm 3} should be $(\mathcal{X}_{n+2},
\mathcal{Y}_{n+2})$ and $\tau_{c}=n_ch$. For convenience, we  still use  $p^{(x)}$ and
$p^{(y)}$ to denote the respective probability density functions of
$\mathcal{X}_{n+2}$ and $\mathcal{Y}_{n+2}.$ In this way, the two-step version of {\bf
  Algorithm 3} is as follows: (i) Sample $\mathcal{X}_{n+2}$ and calculate $W
= U p^{(x)}( \mathcal{X}_{n+2})$; (ii) If $W \leq p^{(y)}(
\mathcal{X}_{n+2})$,  let $\mathcal{X}_{n+2} =\mathcal{Y}_{n+2}, \tau_{c} = (n+2)h.$  Otherwise,  sample $\mathcal{Y}_{n+2}$ and calculate $W'
= Vp^{(y)}( \mathcal{Y}_{n+2})$ until $W' > p^{(x)}(
\mathcal{Y}_{n+2})$. This method 
works for other similar problems with degenerate diffusions. If the noise is very degenerate, one may need to calculate the probability density function after more than two steps.

It is not easy to explicitly estimate the
probability density function of the  Euler-Maruyama method for two steps (or more). (One exception is  the Langevin dynamics because the derivative
of the position variable is a linear function of the velocity, which makes it possible to 
calculate an explicit probability density function;  see the
first author's another recent paper  \cite{dobson2019using}.) We need to use the
transformation of probability density functions to calculate  $p^{(x)}$ and
$p^{(y)}$ at different points. Our implementation is as below. 

Let $\bar{S}_{n}$ and $\bar{I}_{n}$ be the approximate values of $S^{h}_{n}$ and
$I^{h}_{n}$ 
when running the Euler-Maruyama method. After one step iteration, we have
\begin{align*}
  \bar{S}_{n+1} &= \bar{S}_{n} + (\alpha - \beta \bar{S}_{n}\bar{I}_{n} - \mu \bar{S}_{n})h + \sigma
            \bar{S}_{n} \sqrt{h} N_{1} := \widetilde{S}_{n+1}+ \sigma \bar{S}_{n}
            \sqrt{h}N_{1},\\
\bar{I}_{n+1} &= \bar{I}_{n} + (\beta \bar{S}_{n} \bar{I}_{n} - (\mu + \rho + \gamma) \bar{I}_{n})h +
          \sigma \bar{I}_{n}\sqrt{h}N_{1} :=  \widetilde{I}_{n+1} + \sigma \bar{I}_{n}
          \sqrt{h} N_{1},
\end{align*}
where $N_{1}$ is a standard normal random variable. After two steps, with some calculations  we have
\begin{align}
\label{2steps}
  \bar{S}_{n+2} &=  \widetilde{S}_{n+1} + (\alpha - \beta
             \widetilde{S}_{n+1} \widetilde{I}_{n+1} - \mu  \widetilde{S}_{n+1})h +
            R_{S}(N_{1}, N_{2})\\\nonumber
\bar{I}_{n+2} &=  \widetilde{I}_{n+1} + (\beta  \widetilde{S}_{n+1}  \widetilde{I}_{n+1} - (\mu +
          \rho + \gamma)  \widetilde{I}_{n+1})h + R_{I}(N_{1}, N_{2}),
\end{align}
where $N_{1}$, $N_{2}$ are two independent standard normal random
variables. The transformations $R_{S}$ and $R_{I}$ are as follows
\begin{align}
\label{eqnRS}
  R_{S}(N_{1}, N_{2}) &=  [- \beta \sigma \bar{S}_{n}h^{3/2} \widetilde{I}_{n+1} - \beta \sigma
  \bar{I}_{n}h^{3/2}  \widetilde{S}_{n+1} - \mu \sigma \bar{S}_{n}h^{3/2} + \sigma \bar{S}_{n}
  h^{1/2}] N_{1} \\\nonumber
&+ \sigma  \widetilde{S}_{n+1}h^{1/2} N_{2} - \beta
  \sigma^{2}\bar{S}_{n} \bar{I}_{n}h^{2}N_{1}^{2} + \sigma^{2}\bar{S}_{n}hN_{1}N_{2} 
\end{align}
and
\begin{align}
\label{eqnRI}
  R_{I}(N_{1}, N_{2}) &= [\beta \sigma \bar{S}_{n}h^{3/2}  \widetilde{I}_{n+1} + \beta \sigma
  \bar{I}_{n}h^{3/2}  \widetilde{S}_{n+1} - (\mu + \rho + \gamma) \sigma \bar{I}_{n}h^{3/2} \\\nonumber&+ \sigma \bar{I}_{n}
  h^{1/2}] N_{1} 
+ \sigma  \widetilde{I}_{n+1}h^{1/2} N_{2} - \beta
  \sigma^{2}\bar{S}_{n} \bar{I}_{n}h^{2}N_{1}^{2} + \sigma^{2} \bar{I}_{n}hN_{1}N_{2}.
\end{align}
For $h$ sufficiently small, the transformation $(N_{1}, N_{2}) \mapsto (R_{S},
R_{I})$ is close to a linear transformation since all the coefficients of quadratic terms 
are significantly smaller than that of the linear terms. Hence,
we treat this transformation as invertible when calculating the probability density function. 

By the elementary probability, it is easy to see that the joint
probability density function $p(R_{S}, R_{I})$ is given by 
\begin{equation}
\label{trans}
  p(R_{S}, R_{I}) = |J|^{-1} p^{norm}(\bar{N}_{1}, \bar{N}_{2}),
\end{equation}
where $J$ is the Jacobian matrix of the transformation $(N_{1}, N_{2})
\mapsto (R_{S}, R_{I})$, $p^{norm}$ is
the probability density function of the 2D standard normal random
variable, and $\bar{N}_{1}, \bar{N}_{2}$ are the values of random variables $N_{1}$ and
$N_{2}$ that produce $(R_{S}, R_{I})$. 

Now, let $\mathcal{X}^{h}_{n} = (\bar{S}^{x}_{n}, \bar{I}^{x}_{n})$ and $\mathcal{Y}^{h}_{n} =
(\bar{S}^{y}_{n}, \bar{I}^{y}_{n})$ be the two numerical trajectories that need to be
coupled. Let $p^{x}$ and $p^{y}$ be the probability density functions
of $\mathcal{X}^{h}_{n+2}$ and $\mathcal{Y}^{h}_{n+2},$
respectively. In {\bf Algorithm 3}, we need to compute four probability densities:
$p^{(x)}(\mathcal{X}^{h}_{n+2})$, $p^{(x)}( \mathcal{Y}^{h}_{n+2})$,
$p^{(y)}(\mathcal{X}^{h}_{n+2})$, and $p^{(y)}(
\mathcal{Y}^{h}_{n+2})$. Since  the normal random variables
$N_{1}$ and $N_{2}$ are already known when sampling $\mathcal{X}^{h}_{n+2}$,
$p^{(x)}(\mathcal{X}^{h}_{n+2})$ is given by  \eqref{trans}
directly. For $p^{(x)}( \mathcal{Y}^{h}_{n+2})$, we need to calculate
the ``effective'' $(R^{y}_{S}, R^{y}_{I})$ from 
\eqref{2steps} for $\mathcal{X}^{h}_{n+2}$, which are the ``effective random terms'' for $\mathcal{X}^{h}_{n+2}$ to produce
$\mathcal{Y}^{h}_{n+2}$. This is done by solving the following equations
\begin{align*}
  \bar{S}^{y}_{n+2} &=  \widetilde{S}^{x}_{n+1}+ (\alpha - \beta
             \widetilde{S}^{x}_{n+1} \widetilde{I}^{x}_{n+1}- \mu  \widetilde{S}^{x}_{n+1} )h +
            R^{y}_{S}(N^{y}_{1}, N^{y}_{2})\\\nonumber
\bar{I}^{y}_{n+2} &=  \widetilde{I}^{x}_{n+1} + (\beta \widetilde{S}^{x}_{n+1}  \widetilde{I}^{x}_{n+1}- (\mu +
          \rho + \gamma)  \widetilde{I}^{x}_{n+1})h + R^{y}_{I}(N^{y}_{1}, N^{y}_{2}).
\end{align*}
Then we solve $(N_{1}^{y}, N_{2}^{y})$ by numerically solving
equation \eqref{eqnRS} and \eqref{eqnRI} for $(R^{y}_{S},
R^{y}_{I})$. We use Newton's method which
converges after less than $5$ steps. This gives the ``effective normal random variables'' for
$\mathcal{X}^{h}_{n+2}$ to produce $\mathcal{Y}^{h}_{n+2}$. The
probability density function $p^{(x)}( \mathcal{Y}^{h}_{n+2})$ is obtained
by applying the transformation \eqref{trans} to the numerically solved
$(N_{1}^{y}, N_{2}^{y})$. Computations of $p^{(y)}(\mathcal{X}^{h}_{n+2})$ and $p^{(y)}(
\mathcal{Y}^{h}_{n+2})$ are analogous.

We remark that this is a representative example because many random
dynamical systems in various different settings admit random attractors
\cite{arnold1995random, debussche1997finite, schmallfuss1997random,
  wang2020long, wang2019asymptotic}. This means that any trajectory along 
the same Brownian sample path, denoted by $\omega$, will
converge to an $\omega$-dependent set $A(\omega)$. If $A(\omega)$
is a stable equilibrium, the synchronous coupling can bring any two trajectories
close to each other. It is also called reliability by some authors \cite{lin2009reliability}. When the two trajectories close enough, one can shift to the maximal coupling to make them collapse together. This approach builds 
some additional connections between the theories of random dynamical systems and stochastic differential equations.

\subsection{Coupled stochastic FizHugh-Nagumo model}
A significant  advantage of the coupling method used in this paper is that
it is relatively dimension-free. In contrast, approaches relying
on the discretization of the generator is extremely difficult when dealing
with higher dimensional problems. In this subsection, we consider a very
high dimensional example: the stochastic FizHugh-Nagumo(FHN) model, for which the many stochastically
FHN oscillators are coupled. It is well known that the FHN
model is a nonlinear model that models the periodic evolution of the membrane
potential of a spiking neuron under external stimulations. For a single
neuron, this model is a 2D dynamical system with additive noise
\begin{eqnarray}
\label{FHN1}
 \mu \mathrm{d}u& = &(u - \frac{1}{3}u^{3} - v) \mathrm{d}t +
                           \sqrt{\mu} \sigma \mathrm{d}W^{(1)}_{t}
  \\\nonumber
\mathrm{d} v &=& (u + a) \mathrm{d}t + \sigma \mathrm{d}W^{(2)}_{t},
\end{eqnarray}
where $u$ represents the membrane potential, $v$ is a recovery
variable, and $W^{(1)}_{t}, W^{(2)}_{t}$ are two independent Brownian motions. When $a = 1.05$, the deterministic system admits a stable
equilibrium with a small basin of attraction. Intermittent limit cycles
can be triggered by suitable random perturbations which are strong
enough to drive the system out from the basin of attraction.

Consider $50$ coupled equations \eqref{FHN1} with both the nearest-neighbor
interaction and a mean-field interaction. Similar as in \cite{chen2019spatial}, 
let $v = \sqrt{\mu}v$ be the new recovery variable. This gives the
coupled FHN oscillator
\begin{align}
\label{FHNc}
 \mathrm{d}u_{i}& = \left(\frac{1}{\mu} u_i - \frac{1}{3\mu} u_i^{3} -
                  \frac{1}{\sqrt{\mu}}v_i + \frac{d_{u}}{\mu} ( u_{i+1}
                      + u_{i-1} - 2 u_{i}) + \frac{w}{\mu}( \bar{u} - u_{i})\right ) \mathrm{d}t +
                           \frac{\sigma}{\sqrt{\mu}} \mathrm{d}W^{(2i-1)}_{t}
  \\\nonumber
\mathrm{d} v_{i} &=(\frac{1}{\sqrt{\mu}} u_i + \frac{a}{\sqrt{\mu}}) \mathrm{d}t + \frac{\sigma}{\sqrt{\mu}}\mathrm{d}W^{(2i)}_{t} \,
\end{align}
for $i = 1,\cdots, 50$, where $d_{u}$ is the neareast-neighbor coupling
strength, $w$ is the mean field coupling strength, $W^{(1)}_{t},
\cdots, W^{(100)}_{t}$ are independent Brownian motions, and
$$
  \bar{u} = \frac{1}{50} \sum_{i = 1}^{50} u_{i} 
$$
is the mean membrane potential. We set $u_{0} =
u_{50}$ and $u_{51} = u_{1}$ so that the $50$ neurons are
connected as a ring. We would like to use this example to demonstrate
the strength  of our algorithm when dealing with the high-dimensional problems. The
connection between the ergodicity and degree of synchrony will also be
discussed. 

In our simulations, we choose parameters $w = 0.4$, $\mu = 0.05$, and
$\sigma = 0.6$. These parameters are similar to those in
\cite{chen2019spatial}. The main control parameter is $d_u$. A higher $d_u$ means a stronger nearest-neighbor coupling, which gives a more synchronized
dynamics. See Figure \ref{FHNfig} Panel I-V for the time evolutions
of the membrane potential with different $d_{u}$. We  see that a higher  $d_u$ makes the membrane potentials of the $50$ neurons evolve more coherently. 

\begin{figure}[htbp]
\centerline{\includegraphics[width = 1.3\linewidth]{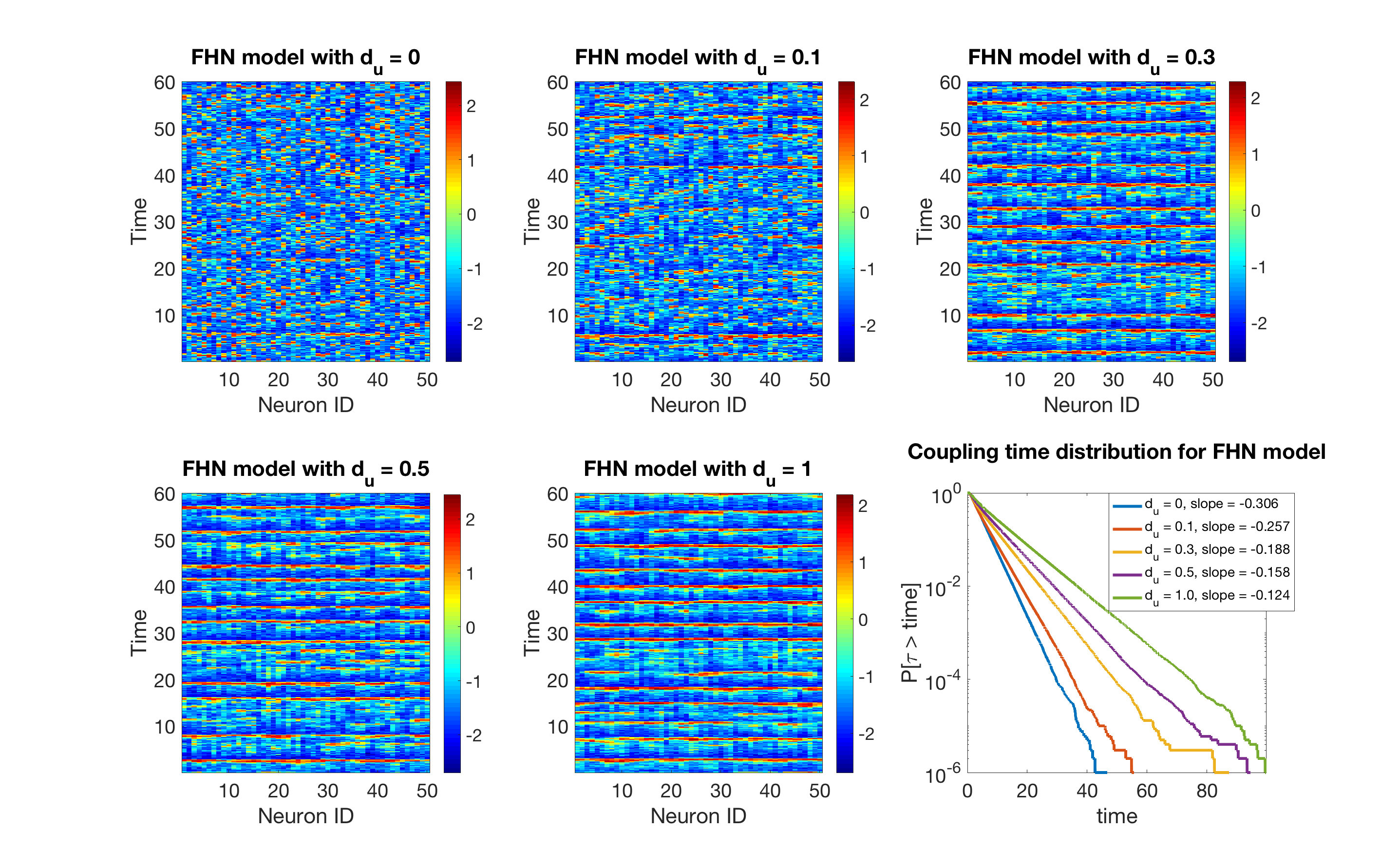}}
\caption{Panel I-V: Time evolutions of membrane potential of $50$
  coupled neurons in FHN model. Coupling strength $d_{u}$ takes value $0,
  0.1, 0.3, 0.5$, and $1$ in five figures. Different colors means different
  membrane potentials (see the color bar beside). X-axis: Neuron ID. Y-axis:
  Time. Panel VI: Coupling time distributions of FHN model with five different
  $d_u$ values in a log-linear plot.}
\label{FHNfig} 
\end{figure}

We use the Euler-Maruyama scheme  in our simulations  with the 
 step size $h = 0.001$. We run {\bf Algorithm 2} with $N=10^{6}$ samples for
$ d_u = 0, 0.1, 0.3, 0.5$, and $1$ to compute the slopes of exponential tails of
the distribution of coupling times. See Figure \ref{FHNfig} Panel VI
for a comparison of the coupling time distributions and slopes. We see that higher $d_{u}$'s provide longer
coupling times, and hence lower rates of  geometric ergodicity. Heuristically,
this phenomenon is caused by the phase lock. In the presence of strong
synchronization, the trajectories are attracted to the neighborhood of a
high dimensional limit cycle and follow it as time evolves. When
running the coupling process, the two independent trajectories can be
attracted to difference phases of this limit cycle. When this happens,
it  will take longer times for the two trajectories to couple, as one
trajectory needs to diffuse by itself to ``chase'' the other one along the limit cycle.

\section{Conclusion and further discussions}
The geometric ergodicity is an important property of a stochastic process
with an infinitesimal generator. It measures the mixing effect given by a
combination of the underlying deterministic dynamics and the random
perturbations. In this paper, based on the coupling
technique, we propose a probabilistic method to numerically compute the rate of
geometric ergodicity.   Some straightforward arguments show that the lower bound of the rate can be estimated by computing the exponential
tail of the coupling times. In addition, we find that the upper bound of the
geometric convergence rate can also be estimated by computing the first
exit time with respect to a sequence of disjoint sets pairs. Compared with the traditional method that looks for the eigenvalues of the discretized infinitesimal generator, our method is relatively dimension-free. It works well when the dimension of the
phase space becomes too high for the grid-based method to handle.

As numerical examples, we study several deterministic dynamical
systems with additive noise perturbations. One interesting finding is
that the coupling time distributions under  noise magnitudes can provide a lot of information about the deterministic
dynamics. As demonstrated in Section 4, the random perturbed systems admit different convergence rate versus noise curves when their
underlying deterministic dynamics admit different degrees of chaos. In other words, the coupling times provide some data-driven
inference of the underlying deterministic dynamics. Since the coupling method is
relatively dimension-free, we expect that this approach can be used to characterize some high-dimensional deterministic dynamical systems, such as the gradient
flows of high-dimensional potential functions. We plan to further
explore along this direction in future works. 

Despite the success of the many examples, the coupling method has its own limitations. Although there are some known results about coupling with degenerate noise, such as the coupling for the Langevin
dynamics \cite{eberle2019couplings} or the Hamiltonian Monte Carlo method \cite{bou2018coupling}. When the noise
is highly degenerate, it becomes difficult to design an effective
coupling scheme. In addition, with degenerate noise, the numerical maximal
coupling updates become significantly  difficult, as
one needs to compute the probability density function of several
consecutive updates in order to get a non-degenerate probability density
function. As shown in Section 5.4, even the implementation of a
relatively simple 2D example has some nontrivial overhead. At each step,
one needs to run a nonlinear equation solver twice to check the probability of coupling. In this
situation, a ``weaker'' approach based on the numerical return
time and analytical minorization condition works better; see the first author's
earlier paper \cite{li2017numerical}. The method in \cite{li2017numerical}
can numerically check the qualitative rate of ergodicity (geometric or
sub-geometric), although in general it does not give a useful bound for the
rate of geometric ergodicity. The first author is currently writing a
separate paper to extend the method in \cite{li2017numerical} to the case of
SDEs with highly degenerate noise terms.

\section*{Acknowledgement}
The authors would like to thank the referees for their valuable and  constructive comments which significantly improve the quality of this paper  in both presentation and substance.
Y. L. was partially supported by NSF DMS-1813246. S. W. was  partially supported by NSFC grants 11771026, 11471344, and acknowledges PIMS-CANSSI postdoctoral fellowship.

\bibliography{myref}
\bibliographystyle{amsplain}
\end{document}